\documentclass[12pt]{elsarticle}
\usepackage{amssymb}
\usepackage{amsthm}
\usepackage{amsmath}
\usepackage{graphics}
\newtheorem{theorem}{Theorem}[section]
\newtheorem{lemma}[theorem]{Lemma}
\newtheorem{corollary}[theorem]{Corollary}
\begin{document}
\begin{frontmatter}
\title{Higher order dispersive effects in regularized Boussinesq equation}

 \author[itu]{Goksu Oruc }
 \author[isik]{Handan Borluk\corref{cor1} }
 \ead{handan.borluk@kemerburgaz.edu.tr}
 \cortext[cor1]{Corresponding author}
 \author[itu]{Gulcin M. Muslu}
 \address[itu]{Istanbul Technical University, Department of Mathematics, Maslak 34469,
         Istanbul,  Turkey.}
 \address[isik]{Istanbul Kemerburgaz University, Department of Basic Sciences, Bagcilar 34217,
         Istanbul,  Turkey.}
\begin{abstract}
In this paper,  we consider the higher order Boussinesq (HBq) equation which models the bi-directional propagation of longitudinal waves in various continuous media. The equation contains the higher order effects of frequency dispersion. The present study is devoted to the numerical investigation of the HBq equation.  For this aim a numerical scheme combining
 the Fourier pseudo-spectral method in space  and a Runge Kutta method in time is  \mbox{constructed}. The convergence of semi-discrete scheme is proved in an appropriate Sobolev space. To investigate the
 higher order dispersive effects and nonlinear effects on the solutions of HBq equation,
 propagation of single solitary  wave, head-on collision of solitary waves and blow-up solutions are considered.
\end{abstract}
\begin{keyword}
The higher-order Boussinesq equation  \sep Fourier pseudo-spectral method \sep Solitary waves
 \sep Head-on collision of  solitary waves \sep Blow-up

\MSC[2010] 65M70  \sep 35C07

\end{keyword}
\end{frontmatter}

\renewcommand{\theequation}{\arabic{section}.\arabic{equation}}
\setcounter{equation}{0}
\section{Introduction}

The first attempt to describe the solitary waves of Scott Russell's experiment in 1834,  theoretically, was  made
by Boussinesq \cite{bouss1, bouss2, bouss3}. To model the motion of the long waves in shallow water he proposed the Boussinesq equation
\begin{equation}
u_{tt}=u_{xx}-\frac{3\beta}{2}(u^2)_{xx}+\frac{\beta}{3}u_{xxxx}. \label{bossinesq}
\end{equation}
We refer the reader to \cite{christov1, christov2} for a more detailed information on the origins of Boussinesq equation.
As the  equation \eqref{bossinesq}  turned out to be incorrect
in the sense of Hadamard, several re-derivations are made.
The most common well-posed re-derivation is proposed  for ion-sound waves and for quasi-continuum approximation of dense lattices \cite{bogolubsky, rosenau1} and is  called the "improved" Boussinesq equation (IBq)
\begin{equation}
u_{tt}=u_{xx}+u_{xxtt}+(f(u))_{xx} \label{boussinesq}.
\end{equation}
Unlike the Boussinesq equation \eqref{bossinesq}, the IBq equation is not fully integrable.
The Boussinesq equation and its re-derivations include the lowest order effects of frequency dispersion. To include the higher order effects of dispersion,  Rosenau \cite{rosenau} proposed the higher order Boussinesq (HBq) equation
\begin{equation}\label{hbq1}
u_{tt}= u_{xx}+ \eta_1u_{xxtt}-\eta_2u_{xxxxtt}+(f(u))_{xx}
\end{equation}
where  $\eta_1$ and $\eta_2$ are real positive constants.
In Rosenau's own words "lowest order dispersion shows some anomalous which can be resolved only if higher order discrete effects are included."
A more recent derivation of the HBq equation is given by Duruk et al. in \cite{duruk2} to model the bi-directional propagation of longitudinal waves in an infinite, nonlocally elastic medium. Eliminating the effects of the higher order dispersion, the HBq equation reduces to the IBq  equation.

\par
The local and global well-posedness of the Cauchy problem for the HBq equation together with the initial conditions
\begin{equation}
u(x,0)=\phi(x), \hspace*{20pt} u_t(x,0)=\psi(x) \label{hbq2}
\end{equation}
was proved  in \cite{duruk2} in the Sobolev space $H^s$ with any $s>1/2$.
The same authors also studied  the qualitative properties of a   general class of nonlocal equations which includes the HBq equation as a special case in \cite{duruk1}.
\par
The present study addresses the answers of the following very natural questions:
How the higher-order dispersive term and the various power type nonlinear terms, i.e. $f(u)=\pm u^p,~~p>1$,  affect the  solutions?
Do  solutions  of  the  HBq  equation converge to solutions of the generalized IBq equation as $\eta_2\rightarrow 0^+$?
To answer these questions we need an efficient numerical method. We therefore propose a numerical method combining a Fourier pseudo-spectral method for the space discretization and a fourth-order Runge-Kutta scheme for time discretization.
To the best of our knowledge, there is no  numerical study for HBq equation although there are lots of numerical studies  to solve the generalized IBq equation (see \cite{borluk} and references therein).
In Section 2,
we  review some properties of the HBq equation and then we derive the solitary wave solutions.
In Section 3, the semi-discrete scheme is introduced for the HBq equation and the convergence
of the semi-discrete scheme is proved in an appropriate space.
In Section 4, we propose the fully-discrete Fourier pseudo-spectral scheme  and show how to formulate it for the HBq equation.
In Section 5, the numerical scheme is tested for accuracy and convergence rate. The effect of extra dispersion and the nonlinearity on the solutions of the HBq equation by considering various problems, like propagation of single solitary wave, head-on collision of two solitary waves and the blow-up solutions are discussed in Section 6.

\renewcommand{\theequation}{\arabic{section}.\arabic{equation}}
\setcounter{equation}{0}
\section{Properties of the higher-order Boussinesq equation}

\par In this section, we  will review the conservation laws of  the HBq equation and we will derive the solitary
wave solution of the HBq equation by using ansatz method.

\par (a) {\it Conservation laws:} Three conserved integrals in terms of $U$ where $u=U_x$  for the
HBq equation are given as
\begin{eqnarray}
 I_1(t)&=&\int_{-\infty}^{\infty} U_t~dx, \label{mass} \\
 I_2(t)&=&\int_{-\infty}^{\infty} U_x[U_t-\eta_1 U_{xxt}+\eta_2 U_{xxxxt}]~dx  \label{momentum} \\
 I_3(t)&=&\int_{-\infty}^{\infty} [(U_t)^2+2F(U_x)+(U_x)^2+\eta_1 (U_{xt})^2+\eta_2 (U_{xxt})^2]~dx \label{energy}
\end{eqnarray}
with $f(s)=F^\prime (s)$. The derivation of the integrals can be found in \cite{rosenau,duruk1}.

\par (b) {\it Solitary wave solution:} We use the ansatz method which is the most effective direct method to
construct the solitary wave solutions of the nonlinear evolution equations. We look for the solutions
of the form $u=u(\xi)$ where $\xi=x-ct-x_0$ under asymptotic boundary conditions.
Substituting the solution $u=u(\xi)$ into equation \eqref{hbq1} and then integrating twice with respect to $\xi$, we have
\begin{equation}
(c^2-1)u-\eta_1 c^2 \frac{d^2 u}{d\xi^2}+\eta_2 c^2 \frac{d^4 u}{d\xi^4}=u^p. \label{hode}
\end{equation}
We now look for the solution of the form
\begin{equation}
u(\xi)=A \mbox{sech}^{\gamma} (B\xi). \label{ansatz}
\end{equation}
If we substitute the above  ansatz into the equation  \eqref{hode},  the solitary wave solution of the HBq equation is  given in the form
\begin{eqnarray} \label{soliter}
&&\hspace{80pt}u(x,t)=A\left\{ {\mbox{sech}}^4\left( B(x-ct-x_0)\right) \right\}^{\frac{1}{p-1}},\hspace{-8pt}
\\  && \hspace{-40pt} A=\left[\frac{\eta_1^2 c^2(p+1)~(p+3)~(3p+1)}{2\eta_2~(p^2+2p+5)^2}\right]^{\frac{1}{p-1}},
 \hspace{10pt} B=\left[\frac{\eta_1(p-1)^2}{4\eta_2(p^2+2p+5)}\right]^{\frac{1}{2}},
\\ &&\hspace{50pt} c^2=\left\{{1-\left[\frac{4\eta_1^2~(p+1)~^2}{\eta_2~(p^2+2p+5)^2}\right]}\right\}^{-1}
\end{eqnarray}
where $A$ is the amplitude and $B$ is the inverse width of the solitary wave.
Here $c$ represents the  velocity of the solitary wave centered at $x_0$ with $c^2>1$.

\setcounter{equation}{0}
\section{The semi-discrete scheme}
\setcounter{equation}{0}
\subsection{Notations and Preliminaries}
Throughout this section, $C$ denotes a generic constant.
We use $(\cdot, \cdot)$ and $\|\cdot\|$ to denote the inner product and the norm of  $L^2(\Omega)$
defined by
\begin{equation}
(u,v)=\int_\Omega u(x)v(x)dx, \hspace*{20pt} \|u\|^2=(u,u)
\end{equation}
for $\Omega=(-L,L)$, respectively. Let  $S_N$ be the space of trigonometric polynomials of degree $N/2$ defined as
\begin{equation}
S_N=\mbox{span}\{e^{{ik\pi x}/{L} } | -N/2\leq k\leq N/2-1 \}
\end{equation}
where $N$ is  a positive integer. $P_N:L^2(\Omega)\rightarrow S_N$ is an orthogonal projection
operator
\begin{equation}
P_Nu(x)=\sum_{k=-N/2}^{N/2-1} \hat{u}_k~ e^{{ik\pi x}/{L}}
\end{equation}
such that for any $u\in L^2(\Omega) $
\begin{equation}
(u-P_Nu,\varphi)=0, \hspace*{20pt} \forall \varphi \in S_N.
\end{equation}
The projection operator $P_N$ commutes with derivative in the distributional sense:
\begin{equation}\label{commute}
D_x^n P_Nu=P_ND_x^nu \hspace*{20pt} \mbox{and} \hspace*{20pt} D_t^n P_Nu=P_ND_t^n u.
\end{equation}
Here $D_x^n$ and $D_t^n$ stand
for the $n$th-order classical partial derivative with respect to $x$ and $t$, respectively. $H^s_p(\Omega)$ denotes the periodic Sobolev space equipped with the norm
\begin{equation}
\|u\|_s^2= \sum_{k=-\infty}^{\infty} (1+|k|^{2s})|\hat{u}_k|^2
\end{equation}
where $\hat{u}_k=\displaystyle\frac{1}{2L} \int_\Omega u(x)e^{{ik\pi x}/{L}} dx$.
The Banach space $X_s=C^1([0,T];H^s_p(\Omega) )$
is the space of all continuous functions in $H^s_p(\Omega)$ whose distributional derivative  is also in $H^s_p(\Omega)$, with norm
$~\left\Vert u \right\Vert_{X_s}^2=\displaystyle \max_{ t \in [0,T]}(\left\Vert u(t) \right\Vert_{s}^2+\left\Vert u_t (t) \right\Vert_{s}^2)$. In order to prove the convergence of semi-discrete scheme, we need  following lemmas.

\begin{lemma}\cite{canuto, rashid} \label{lemma}
For any real $0\leq \mu \leq s$, there exists a constant $C$ such that
\begin{equation}
\|u-P_N u\|_\mu \leq C N^{\mu-s} \|u\|_s,  \hspace*{20pt} \forall u\in H^s_p(\Omega).
\end{equation}
\end{lemma}
\begin{lemma}\cite{runst} \label{lemma2}
Assume that $f\in C^k(\mathbb{R})$, $u,v\in H^s(\Omega)\cap L^\infty (\Omega)$ and $k=[s]+1$, where $s\geq 0$. Then we have
\begin{equation}
\|f(u)-f(v)\|_s \leq C(M) \|u-v\|_s
\end{equation}
if $\|u\|_\infty \leq M, \|v\|_\infty \leq M, \|u\|_s \leq M$  and $\|v\|_s \leq M$, where $C(M)$ is a constant dependent
on $M$ and $s$.
\end{lemma}

\noindent
\begin{corollary} \label{cor} Assume that $f\in C^3(\mathbb{R})$ and  $u,v\in H^2(\Omega) \cap L^\infty (\Omega)$ then
\begin{equation}
\|f(u)-f(v)\|_2 \leq C \|u-v\|_2
\end{equation}
where the constant C depends on $\|u\|_\infty , \|v\|_\infty $ and $\|u\|_2 , \|v\|_2$.
\end{corollary}

\subsection{Convergence of the semi-discrete scheme}
The semi-discrete Fourier  Galerkin approximation to (\ref{hbq1})-(\ref{hbq2}) is
\begin{eqnarray}
&&u^N_{tt}= u^N_{xx}+ \eta_1 u^N_{xxtt}- \eta_2 u^N_{xxxxtt}+P_Nf(u^N)_{xx}, \label{shbq1}  \\
&& u^N(x,0)=P_N\phi(x), \hspace*{20pt} u^N_t(x,0)=P_N\psi(x) \label{shbq2}
\end{eqnarray}
where $u^N(x,t)\in S_N$ for $0\le t\leq T$. We now state  our main result.

\begin{theorem}
Let  $s\geq 2$ and $u(x,t)$ be the solution of the $2L$ periodic initial value problem \eqref{hbq1}-\eqref{hbq2} satisfying
$u(x,t)\in C^1([0,T];H^s_p(\Omega))$  for any $T>0$ and  $u^N(x,t)$ be the solution of the semi-discrete scheme (\ref{shbq1})-(\ref{shbq2}).
There exists a constant $C$, independent of $N$, such that
\begin{equation}
\|u-u^N\|_{X_2} \leq C(T, \eta_1, \eta_2) N^{2-s}  \|u\|_{X_s}
\end{equation}
for the initial data $\phi,\psi \in H^s_p(\Omega)$.
\end{theorem}
\begin{proof}
Using the triangle inequality, it is possible to write
\begin{equation} \label{u-triangle}
\|u-u^N\|_{X_2}\leq\|u-P_Nu\|_{X_2}+\|P_Nu-u^N\|_{X_2}.
\end{equation}

\noindent
Using Lemma \ref{lemma}, we have the following estimates
\begin{equation}
\|(u-P_Nu)(t)\|_{2}\leq C N^{2-s}\|u(t)\|_{s} \notag
\end{equation}
and
\begin{equation}
\|(u-P_Nu)_t(t)\|_{2}\leq C N^{2-s}\|u_t(t)\|_{s} \notag
\end{equation}
for  $s\geq 2$. Thus, the estimation of the first term at the right-hand side of the inequality \eqref{u-triangle}  becomes
\begin{equation} \label{estimation1}
\|u-P_Nu\|_{X_2}\leq C N^{2-s}\|u\|_{X_s}.
\end{equation}
Now, we need to estimate the second term  $\|P_Nu-u^N\|_{X_2}$ at the right-hand side of the inequality \eqref{u-triangle}.
Subtracting the equation \eqref{shbq1} from \eqref{hbq1}  and taking the inner product with $\varphi\in S_N$ we have
\begin{equation}\label{inner}
\hspace*{-20pt}
\big((u-u^N)_{tt}- (u-u^N)_{xx}- \eta_1 (u-u^N)_{xxtt}+\eta_2 (u-u^N)_{xxxxtt}-(f(u)-P_Nf(u^N))_{xx},~\varphi\big)=0.
\end{equation}
Since
\begin{equation}
\hspace*{-10pt}
\left((u-P_Nu)_{tt}, \varphi\right)=\left((u-P_Nu)_{xx}, \varphi\right)=\left((u-P_Nu)_{xxtt}, \varphi\right)
=\left((u-P_Nu)_{xxxxtt}, \varphi\right)=0
\end{equation}
for all  $\varphi\in S_N$, by \eqref{commute},
 the equation \eqref{inner} becomes
\begin{equation} \label{simpinner}
\begin{split}
\big(\{(P_Nu-u^N)_{tt}- (P_Nu-u^N)_{xx}&-\eta_1 (P_Nu-u^N)_{xxtt}+\eta_2 (P_Nu-u^N)_{xxxxtt} \\
&-(f(u)-f(u^N))_{xx}\},~\varphi\big)=0
\end{split}
\end{equation}
for all $\varphi\in S_N$.
Setting $\varphi=(P_Nu-u^N)_{t}$  in \eqref{simpinner}, using the integration by parts and the spatial periodicity,
a simple calculation shows that
\begin{eqnarray}
&&\left((P_Nu-u^N)_{tt},~(P_Nu-u^N)_{t}\right)=\frac{1}{2}\frac{d}{dt}\|(P_Nu-u^N)_{t}(t)\|^2, \label{norm1}\\
&&\hspace{-11pt}\left((P_Nu-u^N)_{xx},~(P_Nu-u^N)_{t}\right)=-\frac{1}{2}\frac{d}{dt}\|(P_Nu-u^N)_{x}(t)\|^2,\\
&&\hspace{-15pt}\left((P_Nu-u^N)_{xxtt},~(P_Nu-u^N)_{t}\right)=-\frac{1}{2}\frac{d}{dt}\|(P_Nu-u^N)_{xt}(t)\|^2,\\
&&\hspace{-18pt}\left((P_Nu-u^N)_{xxxxtt},~(P_Nu-u^N)_{t}\right)=\frac{1}{2}\frac{d}{dt}\|(P_Nu-u^N)_{xxt}(t)\|^2.  \label{norm3}
\end{eqnarray}
Substituting \eqref{norm1}-\eqref{norm3} in \eqref{simpinner}, we have
\begin{eqnarray} \label{inner2}
&& \hspace*{-40pt} \frac{1}{2}\frac{d}{dt}\left[ \| (P_Nu-u^N)_{t}(t)\|^2 +\|(P_Nu-u^N)_{x}(t)\|^2 + \eta_1 \|(P_Nu-u^N)_{xt}(t)\|^2\right.
\nonumber \\
&& \hspace*{-25pt} +  \left.  \eta_2 \| (P_Nu-u^N)_{xxt}(t) \|^2 \right]
 = \left( (f(u)-f(u^N))_{xx}, (P_N u-u^N)_{t}\right). \notag \\
\end{eqnarray}
In the following, we will estimate the right-hand side of the above equation. Using the Cauchy-Schwarz inequality and  the Corollary \ref{cor}, we have
\begin{eqnarray}\label{right-hand}
\left( (f(u)-f(u^N))_{xx},~(P_Nu-u^N)_{t}\right) &\leq& \|(f(u)-f(u^N))_{xx}\|~\|(P_Nu-u^N)_{t}(t)\|\notag \\
&\leq& \frac{1}{2}\left( \| f(u)-f(u^N)\|_2^2+\|(P_Nu-u^N)_{t}(t)\|^2\right)\notag \\
&\leq& C \left( \|(u-u^N)(t)\|_2^2+\|(P_Nu-u^N)_{t}(t)\|^2\right). \notag \\
\end{eqnarray}
\noindent
Substituting  \eqref{right-hand} in \eqref{inner2} we have
\begin{eqnarray} \label{inner3}
&& \hspace*{-20pt} \frac{1}{2}\frac{d}{dt}\left( \| (P_Nu-u^N)_{t}(t)\|^2 +\|(P_Nu-u^N)_{x}(t)\|^2 +\eta_1 \|(P_Nu-u^N)_{xt}(t)\|^2 \right.  \nonumber \\
&& \hspace*{0.4cm} \left. +\eta_2 \| (P_Nu-u^N)_{xxt}(t) \|^2 \right) \nonumber \\
&&\hspace{2cm}\leq C \left( \|(u-u^N)(t)\|_2^2+\|(P_Nu-u^N)_{t}(t)\|^2\right) \notag \\
&& \hspace{2cm}\leq C \left( \|(u-P_Nu)(t)\|_2^2+ \|(P_Nu-u^N)(t)\|_2^2+ \|(P_Nu-u^N)_t(t)\|^2 \right). \notag \\
\end{eqnarray}
Adding  the terms $\left(P_Nu-u^N,~(P_Nu-u^N)_{t}\right)$ and $\left((P_Nu-u^N)_{xx},~(P_Nu-u^N)_{xxt}\right)$ to both sides of the equation \eqref{inner3}
becomes
\begin{eqnarray}
&& \hspace*{-10pt}
\frac{1}{2} \min\{1, \eta_1, \eta_2\}   \frac{d}{dt}\Big[~
 \|(P_Nu-u^N)(t) \|_2^2 + \| (P_Nu-u^N)_t (t) \|_2^2 ~ \Big]  \notag \\
  && \hspace*{40pt}  \leq C \left( ~\| (u-P_Nu)(t) \|_2^2+\|(P_Nu-u^N)(t)\|_2^2+ \| (P_Nu-u^N)_t (t) \|_{2}^2 ~\right). \notag
\end{eqnarray}
Note that $\displaystyle \| (P_Nu-u^N)(0) \|_{2}=0$ and  $\displaystyle \| (P_Nu-u^N)_t(0) \|_{2}=0$. The Gronwall Lemma and Lemma \ref{lemma} imply that
\begin{eqnarray}
\| (P_Nu-u^N)(t) \|_{2}^2 + \| (P_Nu-u^N)_t (t) \|_{2}^2 & \leq &\int_0^t \| u(\tau)-P_Nu(\tau) \|_2^2~ e^{C(t-\tau)} d\tau\notag  \\
                       & \leq & C(T, \eta_1, \eta_2) ~ N^{4-2s} \int_0^t \| u(\tau)\|_s^2  ~d\tau \notag \\
                         \label{esz}\end{eqnarray}
for $s\geq 2$. Finally, we have
\begin{equation}
\| P_Nu-u^N \|_{X_2} \leq C(T, \eta_1, \eta_2)  N^{2-s}~ \| u\|_{X_s}. \label{estimation2}
\end{equation}

\noindent
Using \eqref{estimation1} and \eqref{estimation2} in \eqref{u-triangle}, we complete the proof of Theorem 2.4.
\end{proof}

\noindent
Note that the convergence theorem is proved  for the semi-discrete Fourier Galerkin approximation.
We point out that  the Fourier collocation  pseudo-spectral method is used in the following section as it is more practical
due to the use of  FFT.

\setcounter{equation}{0}
\section{The fully-discrete scheme}
We solve the HBq  equation by combining a Fourier pseudo-spectral method for the space component
and a fourth-order Runge Kutta scheme (RK4) for time.
If the spatial period $[-L,L]$ is, for convenience,
normalized to $[0,2\pi]$ using the transformation
\mbox{$X=\pi(x+L)/L$}, the equation \eqref{shbq1} becomes
\begin{equation}\label{hbq-zero-2pi}
u^N_{tt}-\left(\frac{\pi}{L}\right)^2u^N_{XX}-\eta_{1}\left(\frac{\pi}{L}\right)^2u^N_{XX\,t\,t}+\eta_{2}\left(\frac{\pi}{L}\right)^4u^N_{XXXX\,t\,t}=\left(\frac{\pi}{L}\right)^2(u^N)^{^p}_{XX}.
\end{equation}
The interval $[0,2\pi]$ is divided into $N$ equal subintervals with grid spacing
$\Delta X=2\pi/N$, where the integer $N$ is even. The spatial grid points are given by
$X_{j}=2\pi j/N$,  $j=0,1,2,...,N$. The approximate solutions to
$u^N(X_{j},t)$ is denoted by $U_{j}(t)$.
The discrete Fourier transform of the sequence
$\{U_{j} \}$, i.e.
\begin{equation}\label{dft}
  \widetilde{U}_{k}={\cal F}_{k}[U_{j}]=
          \frac{1}{N}\sum_{j=0}^{N-1}U_{j}\exp(-ikX_{j}),
           ~~~~-\frac{N}{2} \le k \le \frac{N}{2}-1~.
\end{equation}
gives the corresponding Fourier coefficients. Likewise, $\{U_{j} \}$ can
be recovered from the Fourier coefficients by the inversion formula
for the discrete Fourier transform (\ref{dft}), as follows:
\begin{equation}\label{invdft}
  U_{j}={\cal F}^{-1}_{j}[\widetilde{U}_{k}]=
          \sum_{k=-\frac{N}{2}}^{\frac{N}{2}-1}\widetilde{U}_{k}\exp(ikX_{j}),
          ~~~~j=0,1,2,...,N-1     ~.
\end{equation}
Here $\cal F$ denotes the discrete Fourier transform and
${\cal F}^{-1}$ its inverse. These transforms are  efficiently computed
using a fast Fourier transform (FFT) algorithm. In this study, we use  FFT
routines in Matlab (i.e. fft and ifft).

Applying the discrete Fourier transform to the equation \eqref{hbq-zero-2pi} we get the second order ordinary differential equation
which can be written in the following system
\begin{eqnarray}
&& (\widetilde{U}_k)_t=\widetilde{V}_k\label{hbq fourier1} \\
&& (\widetilde{V}_k)_t= -\frac{(\pi k/L)^2}{1+\eta_{1}(\pi k/L)^2+\eta_{2}(\pi k/L)^4}\,[\,\widetilde{U}_k+(\widetilde{U^p})_k\,]. \label{hbq fourier2}
\end{eqnarray}
In order to handle the nonlinear term we use a pseudo-spectral approximation. That is,  we use the formula  ${\cal F}_k [{(U_j)}^p]$
to compute the $k^{th}$ Fourier component of $u^p$. We use the fourth-order Runge-Kutta method to solve the resulting ODE system
\eqref{hbq fourier1}-\eqref{hbq fourier2} in time.
Finally, we find the approximate solution by using the inverse Fourier transform \eqref{invdft}.

\setcounter{equation}{0}

\section{Validation of the fully-discrete scheme}

The purpose of the present section is to verify
numerically  that {\em (i)} the proposed Fourier pseudo-spectral scheme is
highly accurate,
{\em (ii)} the scheme exhibits the fourth-order convergence in
time and {\em (iii)} the scheme has spectral accuracy in space. The $L_\infty$-error norm is defined as
\begin{equation}
L_{\infty}\mbox{-error}=\max_i |~u_i-U_i~|
\end{equation}

\noindent
where $u_i$ denotes the exact solution at $u(X_i,t)$.

\par
To validate whether the Fourier pseudo-spectral method exhibits the expected
convergence rates in time we perform some numerical experiments.
 In this section we use quadratic nonlinearity where $f(u)=u^2$.
 We take fixed number of spatial grid points $N$ and various values for the number of temporal grid  points $M$.  For $\eta_1=\eta_2=1$,
the initial data cor\-res\-ponding to the solitary wave solution (\ref{soliter}) for quadratic nonlinearity
become as follows:
\begin{eqnarray}
&& u(x,0)=A\, {\mbox{sech}}^4(B~x), \label{hbq-initial1}\\
&& v(x,0)=4\, A\,B\,c\,{\mbox{sech}}^4\left(B~x\right)\,
            \tanh\left( B~x\right). \label{hbq-initial2}
\end{eqnarray}
This initial data  generate a solitary wave initially at \mbox{$x_0=0$}
moving to the right with the amplitude $A \approx 0.39$, speed $c\approx 1.13$ and $B \approx 0.14$.
The problem is solved  for times up to \mbox{$T=5$}. The space interval is chosen as  $-100\le x\le 100$ to eliminate the error due to
boundary effects. In these experiments, we take \mbox{$N=512$} to ensure that
the error due to the spatial discretization is negligible.  The convergence rates  calculated from the $L_{\infty}$-errors
at the terminating time $T=5$ are shown in Table  I. The computed convergence rates agree well
with the fact that Fourier pseudo-spectral method exhibits the  fourth-order convergence in time.

\begin{center}
TABLE I
\end{center}
\vspace*{5pt}
\noindent {\small The convergence rate in time calculated from the
$L_{\infty}$-errors  in the case of single solitary wave
($A \approx 0.39,~~N=512$). }
\begin{center}
\begin{tabular}{|c|cc|} \hline\hline
\hline
$M$&  $L_{\infty}$-error & Order \\
\hline
2     &  8.662E-3   &   - \\
5     &  2.530E-4   &  3.8561  \\
10    &  1.614E-5   &  3.9704 \\
50    &  2.623E-8   &  3.9903  \\
100   &  1.637E-9   &  4.0021  \\
\hline\hline
\end{tabular}
\end{center}
\par To validate whether  the Fourier pseudo-spectral method exhibits the expected
convergence rate in space we perform some further numerical
experiments for various values of $N$ and a fixed value of $M$.
In these experiments we take $M=1000$ to minimize  the temporal errors.  We  present the $L_{\infty}$-errors for the terminating time $T=5$
together with the observed rates of convergence in Table II. These results show that the numerical solution obtained using the
Fourier pseudo-spectral scheme converges rapidly to the accurate solution in space, which is an
indicative of exponential convergence.

\vspace*{10pt}
\begin{center}
TABLE II
\end{center}
\vspace*{5pt}
\noindent {\small The convergence rates in space calculated from
the $L_{\infty}$-errors in the case of the single solitary wave
($A \approx 0.39,~~ M=1000$). }
\begin{center}
\begin{tabular}{|c|lc|} \hline\hline
$N$     & \hspace*{10pt}$L_{\infty}$-error \hspace*{5pt} & \hspace*{5pt} Order   \\
\hline
  10    &    0.211E-1     & ~ -      \\
  50    &    1.747E-3     & ~1.5480  \\
  100   &    4.431E-7     & ~11.9450 \\
  150   &    6.500E-10    & ~16.0916 \\
  200   &    3.884E-13    & ~25.8017 \\
\hline\hline
\end{tabular}
\end{center}
\noindent

\setcounter{equation}{0}
\section{Numerical Experiments}

In this section, we investigate the effect of extra dispersion  term  and the nonlinear term on the solutions of the HBq equation. For this aim, we will consider propagation of a single solitary wave, head-on collision of two solitary waves and blow-up solution.

\subsection{Propagation of a Single Solitary Wave}

First, we study the effect of the nonlinear term on the single solitary wave  solution.
The problem is solved on the space interval \mbox{$-100\le x\le 100$} for times up to \mbox{$T=5$} by using the initial data cor\-res\-ponding to the solitary wave solution \eqref{soliter}. We show the variation of $L_\infty$-errors with $N$ for the HBq equation for  various powers of nonlinearity, namely, $f(u)=u^p$ for $p=2,3,4,5$ in Figure 1. The value of $M$ is chosen to satisfy $\nu=\Delta t / \Delta x=2.56\times 10^{-3}$. We observe that the $L_\infty$-errors decay as the number of grid points increases for various degrees of nonlinearity. Even in the case of the quintic nonlinearity, the $L_\infty$-errors  are  about $10^{-12}$.  We have not seen any numerical results
for the HBq equation in the literature to compare with our  results. This  experiment shows that the proposed method provides  highly accurate numerical results  even for the higher-order nonlinearities.
\begin{figure}[h!]
\includegraphics[width=5.3in]{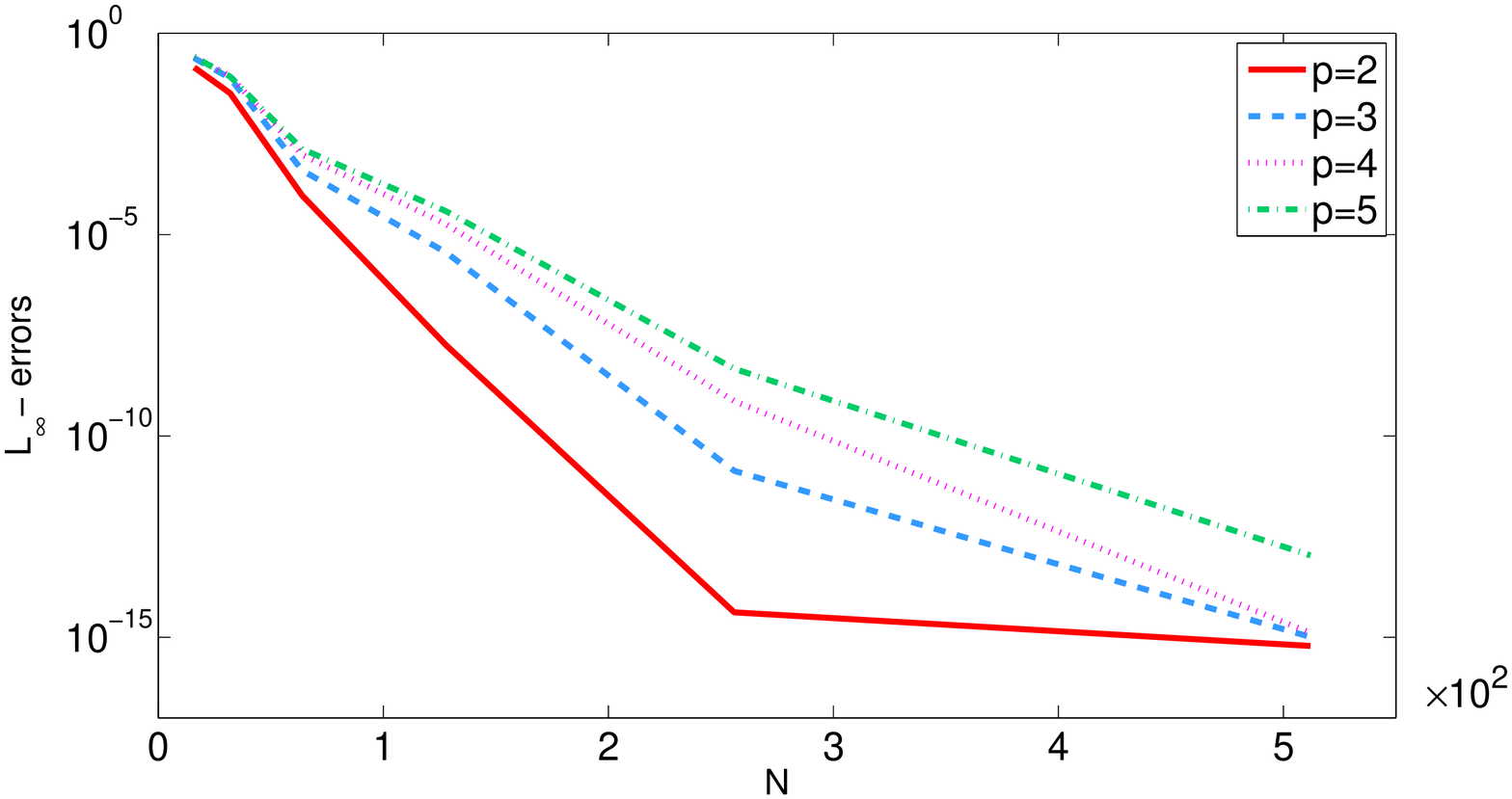}
\caption{ \small{$L_\infty$-errors for the increasing values of $N$. }}
\end{figure}

\par
The question arises naturally how the solutions of the HBq equation
behave when the coefficient of  extra dispersion  term  $\eta_2\rightarrow 0^+$. For this aim, we perform some numerical tests for various values of $\eta_2$
and the fixed value $\eta_1=1$ for quadratic nonlinearity.  We use  the initial data
\begin{eqnarray}
&& u(x,0)=A\, {\mbox{sech}}^2(\frac{1}{c}\sqrt{\frac{A}{6}}~x),  \label{ibq-initial1}\\
&& v(x,0)=2\, A\,{\mbox{sech}}^2\left( \frac{1}{c}\sqrt{\frac{A}{6}}~x\right)\,
            \tanh\left( \frac{1}{c}\sqrt{\frac{A}{6}}~x\right) \label{ibq-initial2}
\end{eqnarray}
where $A=0.9$ and $c=\sqrt{\displaystyle\frac{2A}{3}+1}$. This initial data correspond to the initial profile of the exact solitary wave solution for  the IBq equation given in \cite{borluk}.
 In Figure 2, the solid line shows the exact solution of the
improved Boussinesq equation  which actually corresponds to the HBq equation with $\eta_2=0$.
The dashed lines show the numerical solution of the HBq equation
for the values of $\eta_2=10$, $\eta_2=5$, $\eta_2=1$ and  $\eta_2=0.1$ and $\eta_1=1$.
The problem is solved on the space interval $-100\le x\le 100$ for times up to \mbox{$T=5$} where $N=512$ and $M=5000$. To observe the difference of two solutions more clear, this figure is presented on the space interval $-30\le x\le 30$. The numerical tests indicate that, as the parameter $\eta_2$ tends to zero, the  solutions of the HBq equation converge to the solitary wave solution of the IBq equation.
\newpage
\begin{figure}[h!bt]
\begin{minipage}[t]{0.44\linewidth}
\includegraphics[width=2.5in]{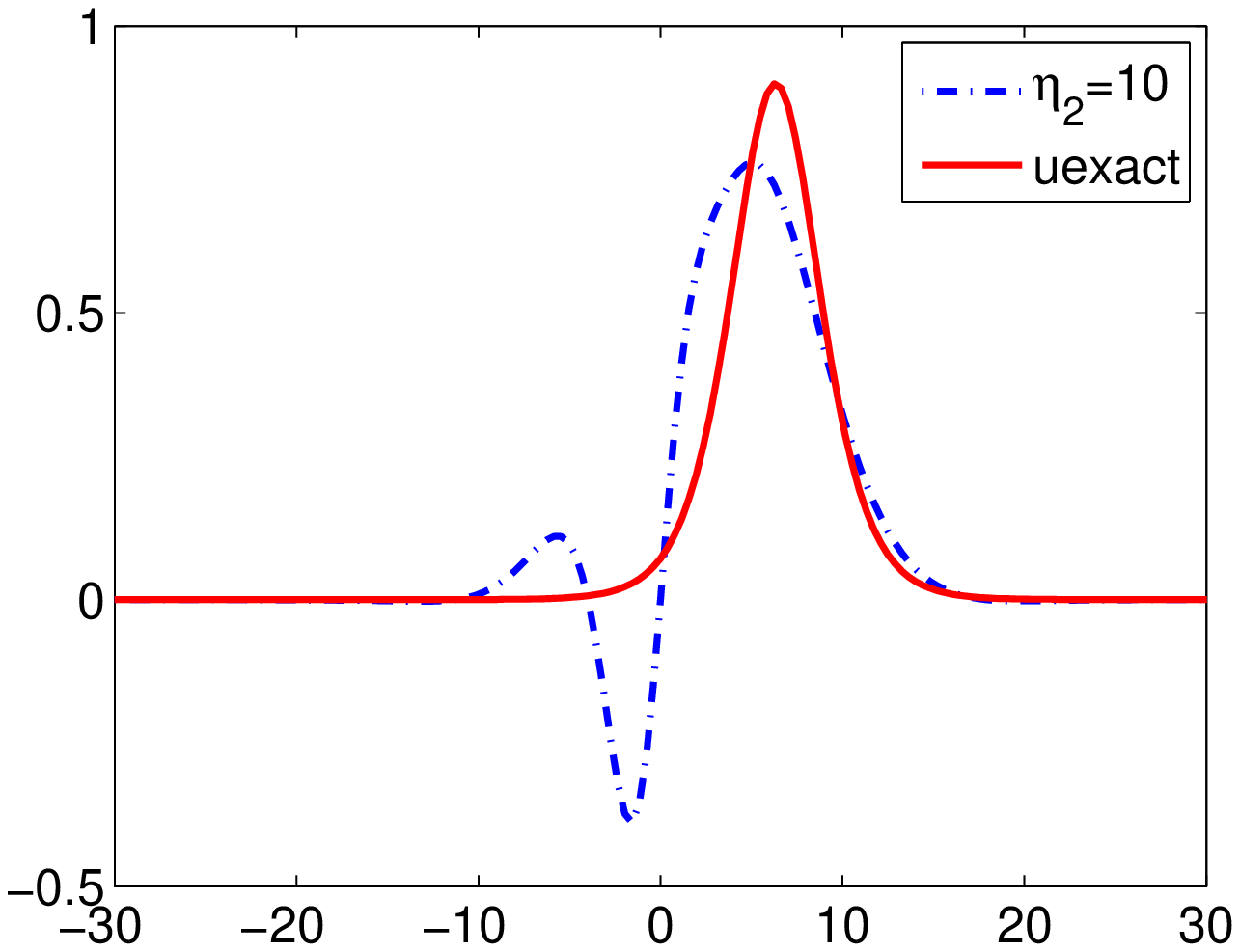}
\end{minipage}
\hspace*{30pt}
\begin{minipage}[t]{0.44\linewidth}
\hspace*{-5pt}
\includegraphics[width=2.5in]{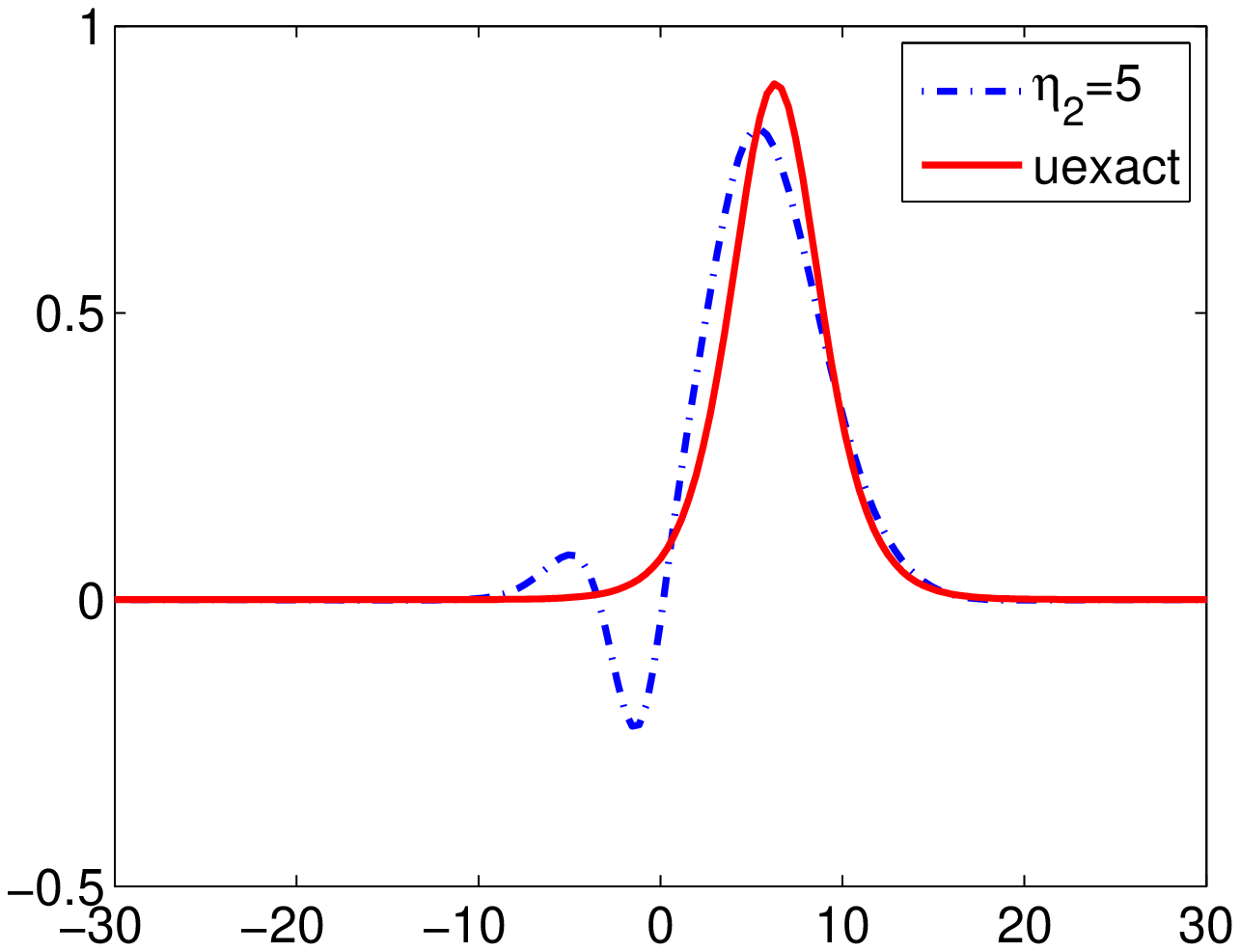}
\end{minipage}
\end{figure}
\begin{figure}[h!bt]

\begin{minipage}[t]{0.44\linewidth}
\includegraphics[width=2.5in]{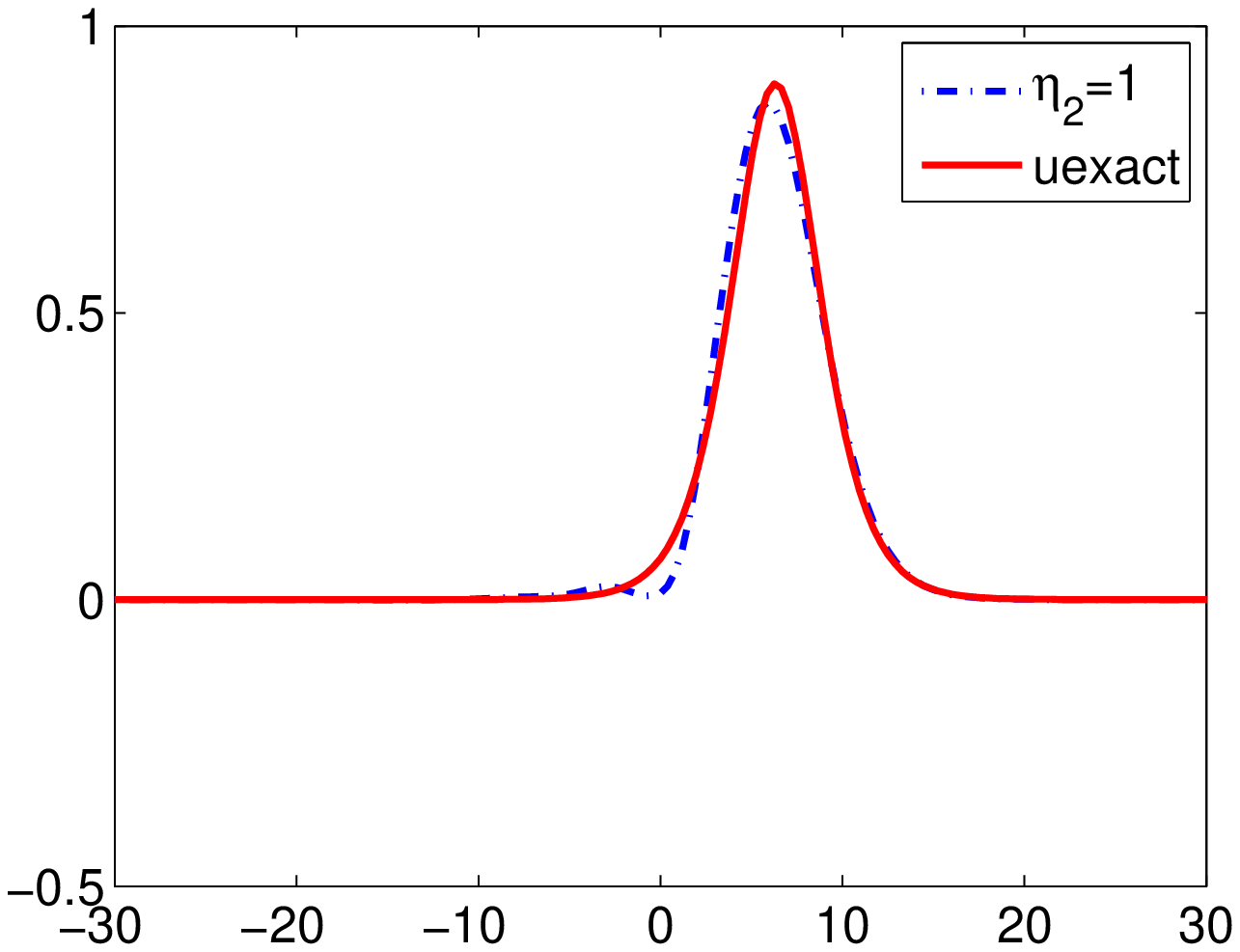}
\end{minipage}
\hspace*{30pt}
\begin{minipage}[t]{0.44\linewidth}
\hspace*{-5pt}
\includegraphics[width=2.5in]{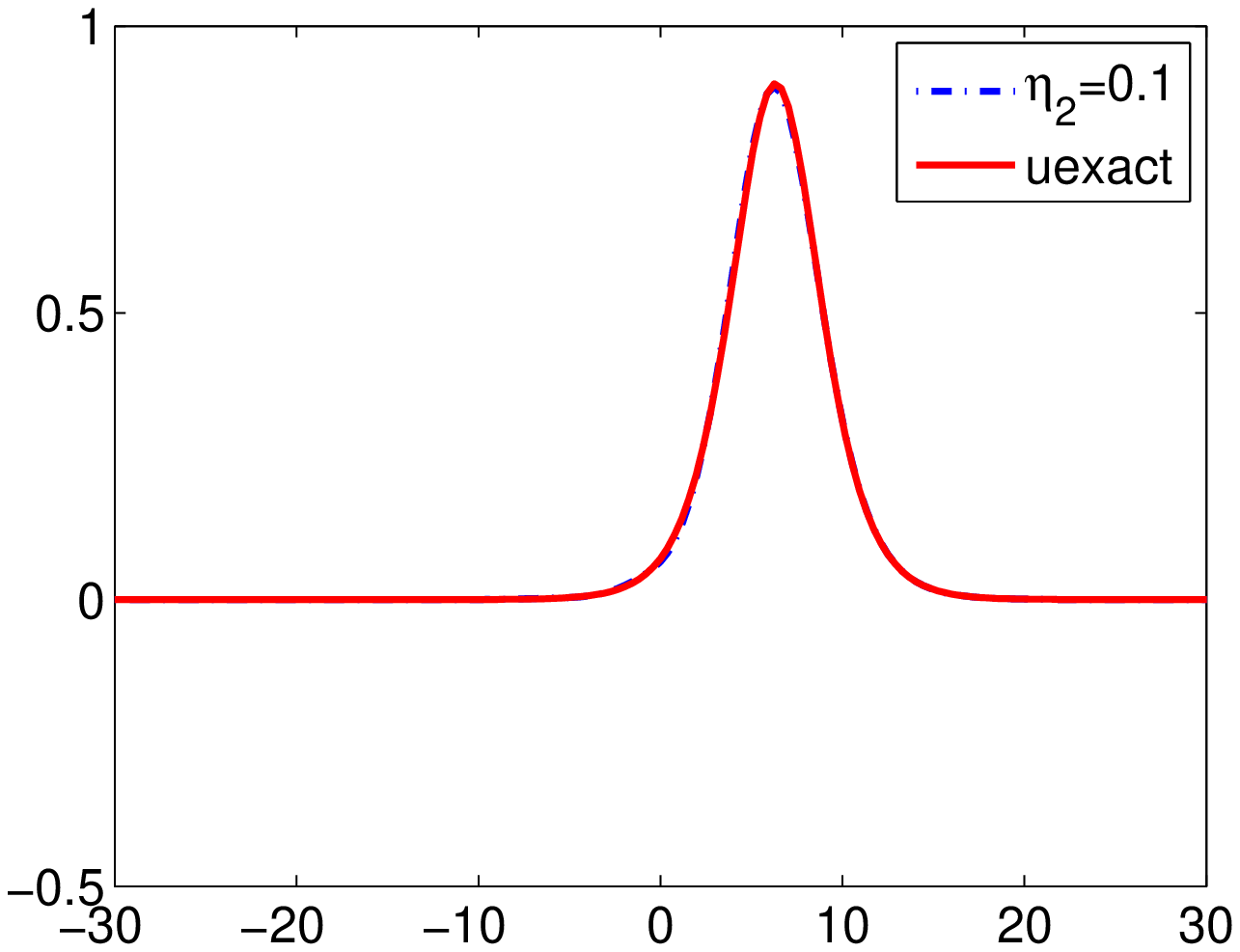}
\end{minipage}
\caption{\small{ The numerical solution of the HBq equation
for  {$\eta_2=10$, $\eta_2=5$, $\eta_2=1$, $\eta_2=0.1$ with $\eta_1=1$ and  the exact solitary wave solution of the IBq equation}}}
\end{figure}
\begin{center}
TABLE III
\end{center}
\vspace*{5pt}
 {\small The amplitude at the final time $T=5$ with respect to changing $\eta_2$. }
\begin{center}
\begin{tabular}{|c|cc|} \hline\hline
\hline
$\eta_2$&   Amplitude  ($p=2$) & Amplitude ($p=3$)  \\
\hline
0.1   &  0.894  &	 0.943\\
0.3    &  0.885 &			0.930 \\
0.5  &  0.876		 &	 0.919\\
0.8   &  0.870		&	 0.912  \\
1     &  0.866 &			0.908 \\
\hline\hline
\end{tabular}
\end{center}

\noindent
Table III shows the dependence of the amplitude to a changing $\eta_2$. In these experiments, we use the initial data \eqref{ibq-initial1}-\eqref{ibq-initial2} with $N=512$ and $M=5000$.
We note that initial amplitude is $0.9$. From Table III, we can draw two observations:

\noindent
{\em (i)} The amplitude of the wave is decreasing with  increasing  values of  $\eta_2$.
This numerical result agrees well with the fact that the wave spreads out with increasing dispersive
effects.

\noindent
{\em (ii)} The amplitude is increasing with increasing values of power of nonlinearity $p$. This numerical result is also agrees well with the  fact that the wave  steepens with increasing nonlinear effects.

\subsection{Head-on Collision of Two Solitary Waves}
\par
In this section, we consider the HBq equation with quadratic nonlinearity
and  we study the head-on collision of two solitary waves with equal amplitudes.
The initial conditions are given by
\begin{eqnarray}
&& \hspace*{-20pt} u(x,0)=\sum_{i=1}^2 A{\mbox{sech}}^4(B(x-x_0^{i})) ,\nonumber \\
&& \hspace*{-20pt} v(x,0)=4\sum_{i=1}^2 ABc_{i} {\mbox{sech}}^4\left(B(x-x_0^{i})\right)\,
            \tanh\left(B(x-x_0^{i})\right). \nonumber
\end{eqnarray}

\noindent
We consider two solitary waves, one initially located at
$x_{0}^1=-40$ and moving to the right with amplitude $A$ ($c_1>0$) and  one initially located at
$x_{0}^2=40$ and moving to the left with amplitude $A$ ($c_2=-c_1$).  The problem is solved again
on the interval $-100\le x\le 100~$ for times up to $T=72$ using the Fourier
pseudo-spectral method. All experiments in this section are performed for $N=512$ and
$M=7200$. We first illustrate the head-on collision of two solitary waves with equal amplitudes
$A \approx 0.39$ with $\eta_1=\eta_2=1$ and  $A \approx 1.08$  with $\eta_1=\eta_2=2$ in Figure 3.

\begin{figure}[h!bt]
\begin{minipage}[t]{0.44\linewidth}
\hspace*{-55pt}
\includegraphics[width=3.4in]{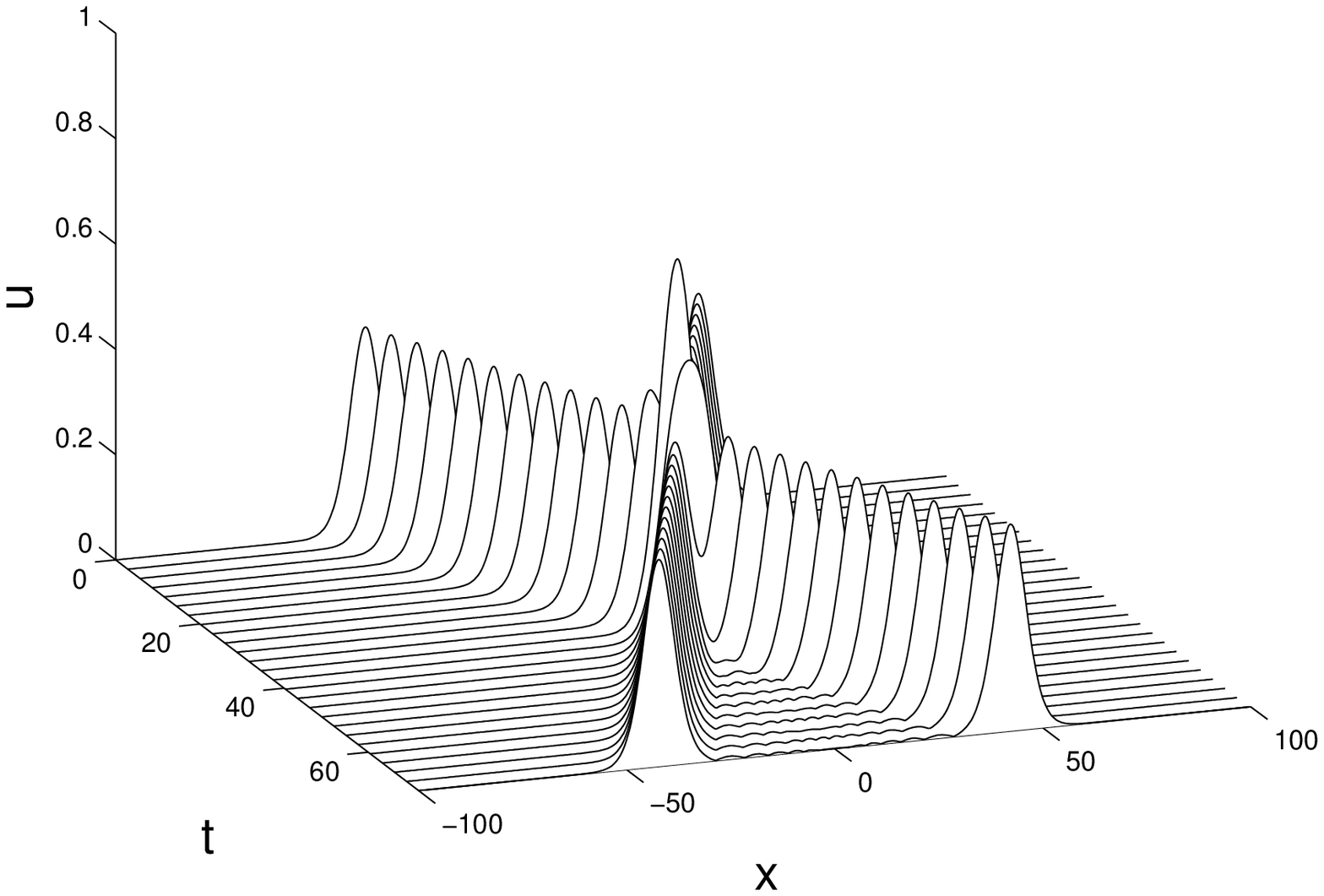}
\end{minipage}
\begin{minipage}[t]{0.44\linewidth}
\hspace*{15pt}
\includegraphics[width=3.4in]{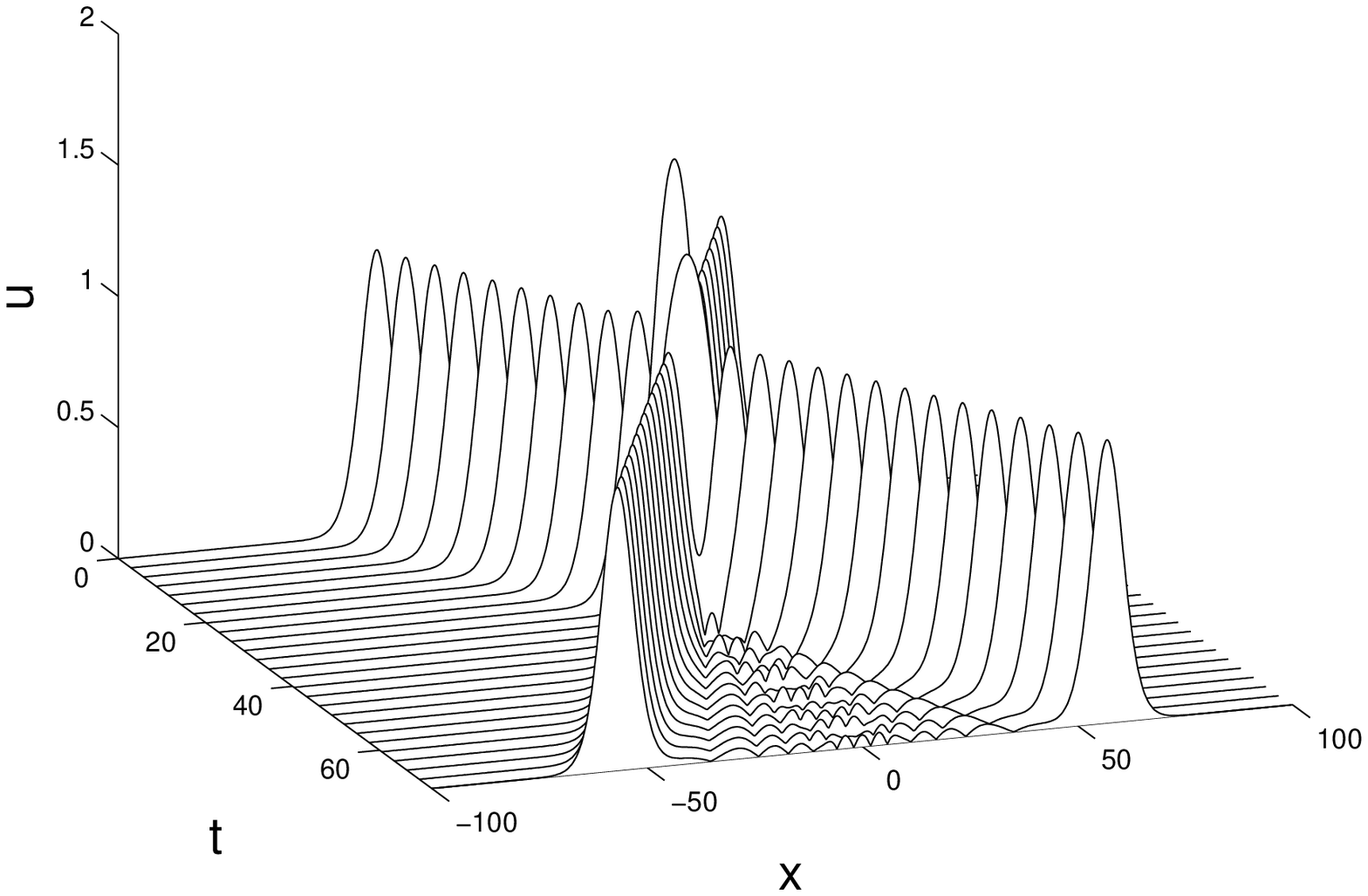}
\end{minipage}
\caption{\small{Surface plot of the head-on collision of two solitary waves with
equal amplitudes. $A \approx 0.39$ (left panel) and $A \approx 1.08$ (right panel).}}
\end{figure}

\par
As  the HBq equation cannot be solved by the inverse scattering method,  the interaction of solitary
waves are inelastic. Although secondary waves exist in all nonlinear interactions, they become more visible as we increase the amplitudes of the interacting waves.

\par
Since an analytical solution is not available for the collision of two solitary waves, we cannot present
the $L_\infty$-errors for this experiment. But, as a numerical check of the proposed Fourier pseudo-spectral scheme, we have observed the evolution of the  change  in  the conserved quantity  $I_1$.
The error $\mid I_1(t)-I_1(0) \mid$ is approximately of order $10^{-14}$ and $10^{-10}$ for the  amplitudes $A \approx 0.39$  and $A \approx1.08$ up to the final time $T=72$, respectively. This behavior provides a valuable check on the numerical results.

\subsection{Blow-up Solution}
In the last section, we investigate the effect of extra dispersion  term  and the nonlinear term on the blow-up solutions of the HBq equation.
The HBq equation is one of the class
of nonlocal equation, we refer to Theorem $5.2$ in \cite{duruk1} as the blow-up criteria. Applied to the HBq equation, this theorem can be restated
as:
\begin{theorem}
    Suppose $\phi=\Phi_x,  \psi=\Psi_x$ for some $\Phi, \Psi \in H^2(\Omega)$ and $F(\varphi) \in L^1(\Omega)$.  If there are
     some
    $\mu >0$ such that
    \begin{equation}
    uf\left( u\right) \leq 2\mu u^2+ 2 \left(1+2\mu \right) F\left( u\right) \mbox{ for all }u\in \mathbb{R}, \label{eszmu}
    \end{equation}
    and
    ${I_3}\left( 0\right) <0~$,    then the solution $u$ of the Cauchy problem for HBq equation with
    the initial conditions $u(x,0)=\phi(x)$, $u_t(x,0)=\psi(x)$ blows up in finite time.
\end{theorem}

\noindent
First, we consider the HBq equation with both quadratic nonlinearity \mbox{$f(u)=u^2$} and cubic nonlinearity \mbox{$f(u)=-u^3$}.  We set the parameters $\eta_1=\eta_2=1$ in \eqref{hbq1}.
For the quadratic nonlinearity, if we choose the
initial data
\begin{equation}
 \phi(x)=4\left(\frac{2x^{2}}{3}-1\right) e^{-\frac{~x^{2}}{3}} , \hspace*{30pt} \psi(x)=\left(x^{2}-1\right)e^{-\frac{~x^{2}}{2}}
 \label{inamp}
\end{equation}
as in \cite{godefroy},  and the condition \eqref{eszmu} is satisfied for $\mu=\frac{1}{4}$.
Similarly, for the cubic nonlinearity, the initial data are chosen as
\begin{equation}
\phi(x)=13 \left(\frac{x^{2}}{2}-1\right)e^{-\frac{~ x^{2}}{4}}, \hspace*{30pt}
\psi(x)=\left(1-x^{2}\right)e^{-\frac{~x^{2}}{2}} \label{inamp2}
\end{equation}
where  the inequality \eqref{eszmu} is  satisfied for $\mu=\frac{1}{2}$. We note that ${I_3}(0)<0$ for both cases. For the experiments in this section, the problem is solved on the interval $-10 \leq x \leq 10$  up to final time $T=4$ for quadratic nonlinearity and  $T=0.4$ for cubic nonlinearity.

\begin{figure}
\begin{minipage}[t]{0.44\linewidth}
\hspace*{-82pt}
\includegraphics[width=4in]{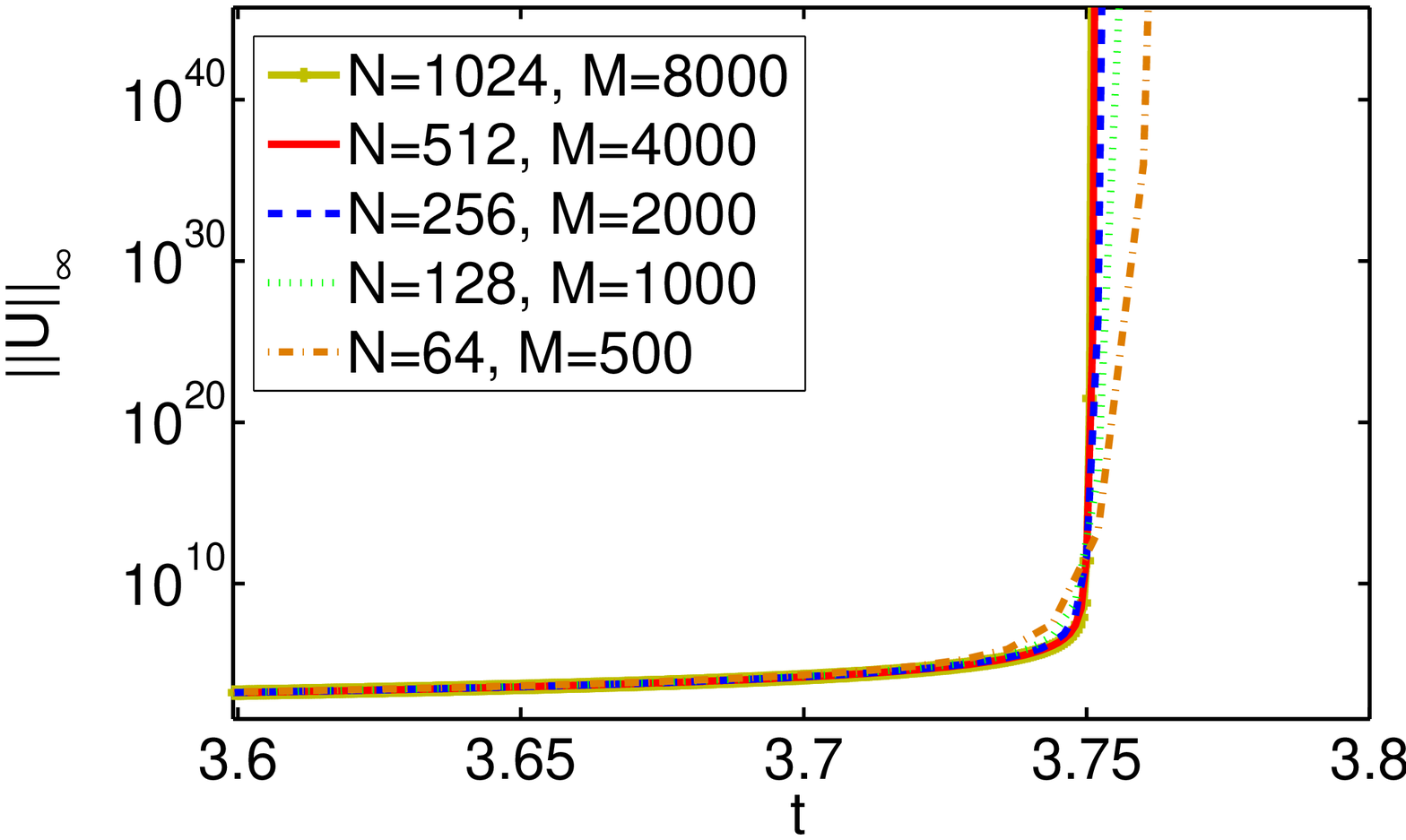}
\end{minipage}
\hspace*{-10pt}
\begin{minipage}[t]{0.44\linewidth}
\includegraphics[width=4in]{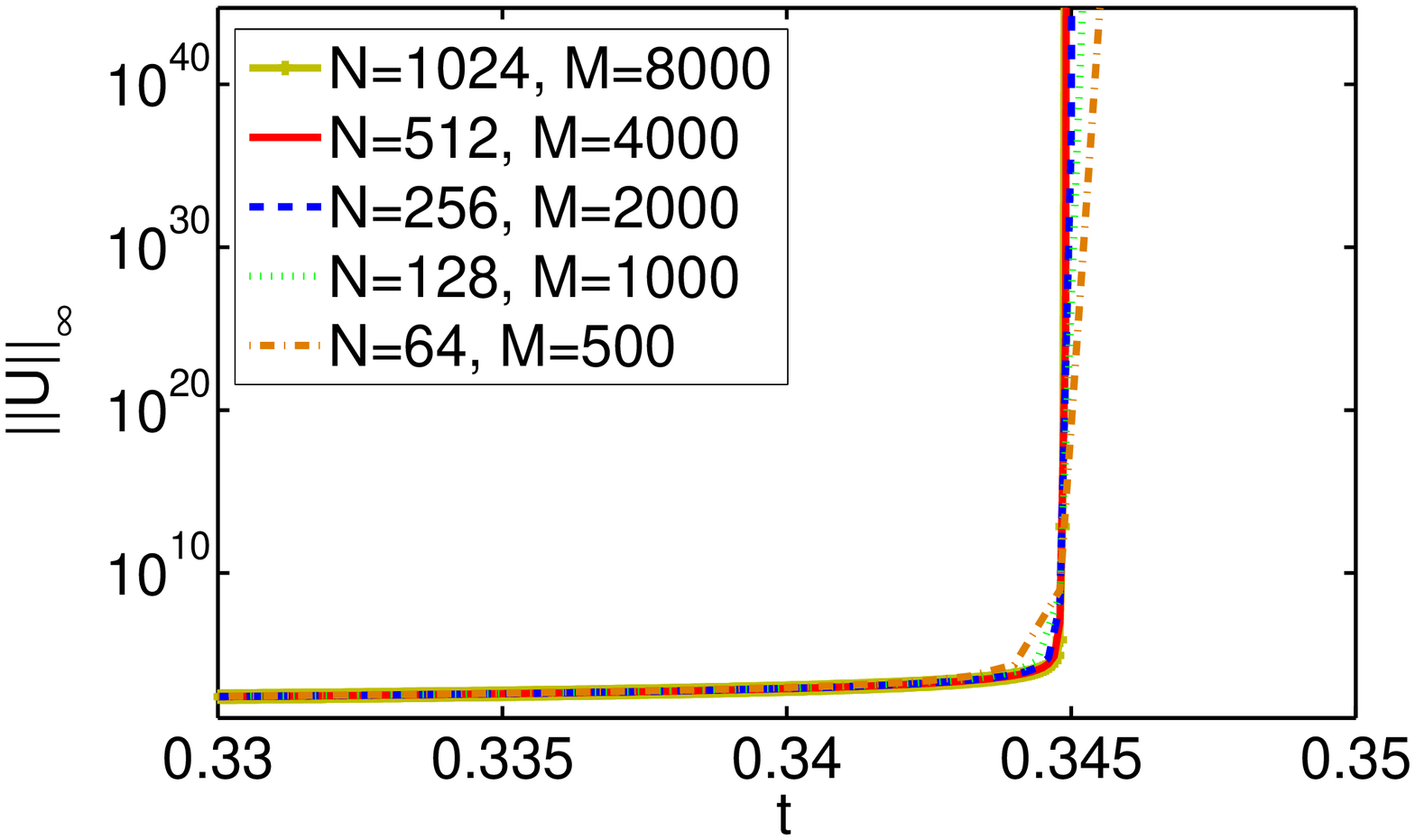}
\end{minipage}

\vspace*{-30pt}
\caption{\small{The variation of $\|U\|_\infty$ with time for increasingly refined discretization
for quadratic nonlinearity (left panel) and  cubic nonlinearity (right panel).}}
\end{figure}

\par
In order to check whether the numerical blow-up results are affected by the time and space step sizes, we now make some numerical experiments with varying resolutions.
We start with the spatial and  temporal resolutions corresponding to $N=64$ and $M=500$, respectively. Then we refine the mesh by increasing the number of  both spatial and temporal grid points. We present  the variation of the $L_\infty$-norm of the approximate solution obtained using the Fourier pseudo-spectral scheme  for various mesh in Figure 4. To observe the blow-up profile more clear, figures  are illustrated near the blow-up time. It is seen that curves  converge to a limiting blow-up profile as the resolution is increased. For  the number of  spatial and temporal grid points more than $N=512$ and $M=4000$, the blow-up profile is not distinguishable from each other. Therefore, we set $N=512$ and $M=4000$ for the rest of the study.
The numerical results strongly indicate that a blow-up  is well underway by time $t=3.8$  for quadratic nonlinearity and $t=0.35$ for cubic nonlinearity. {In Figure 5, we also illustrate the blow-up profile near the blow-up time  at $t=3.710$, $t=3.711$, $t=3.7115$ and $t=3.712$ for the quadratic nonlinearity as in \cite{kalisch}. We observe that the $L_\infty$-norm of the blow-up solution increases very rapidly while  the  blow-up profile remains the same. Therefore, the blow-up appears to be a similarity type.}

\begin{figure}[h!bt]
\begin{minipage}[t]{0.44\linewidth}
\includegraphics[width=2.5in]{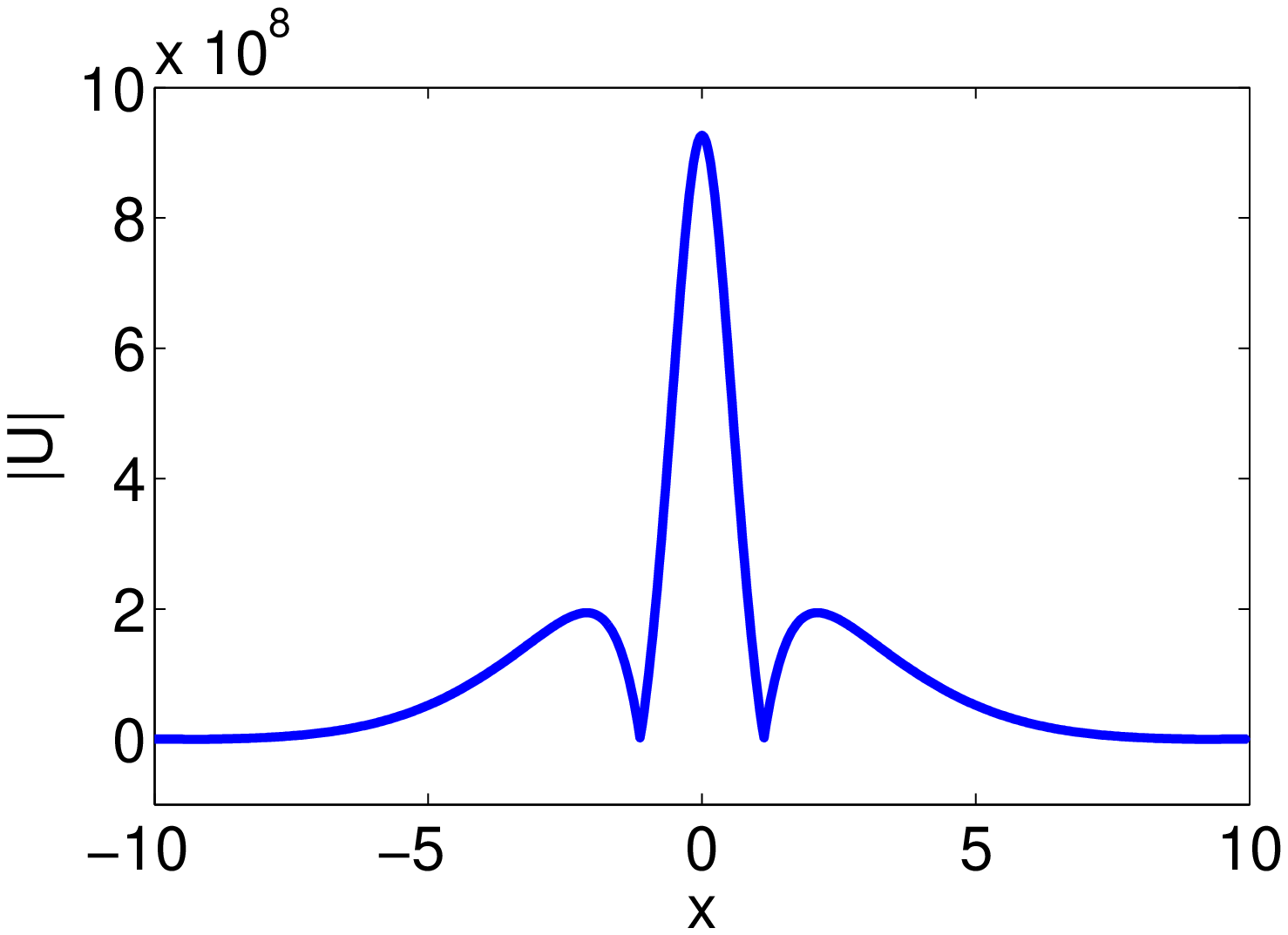}
\end{minipage}
\hspace*{30pt}
\begin{minipage}[t]{0.44\linewidth}
\hspace*{-5pt}
\includegraphics[width=2.5in]{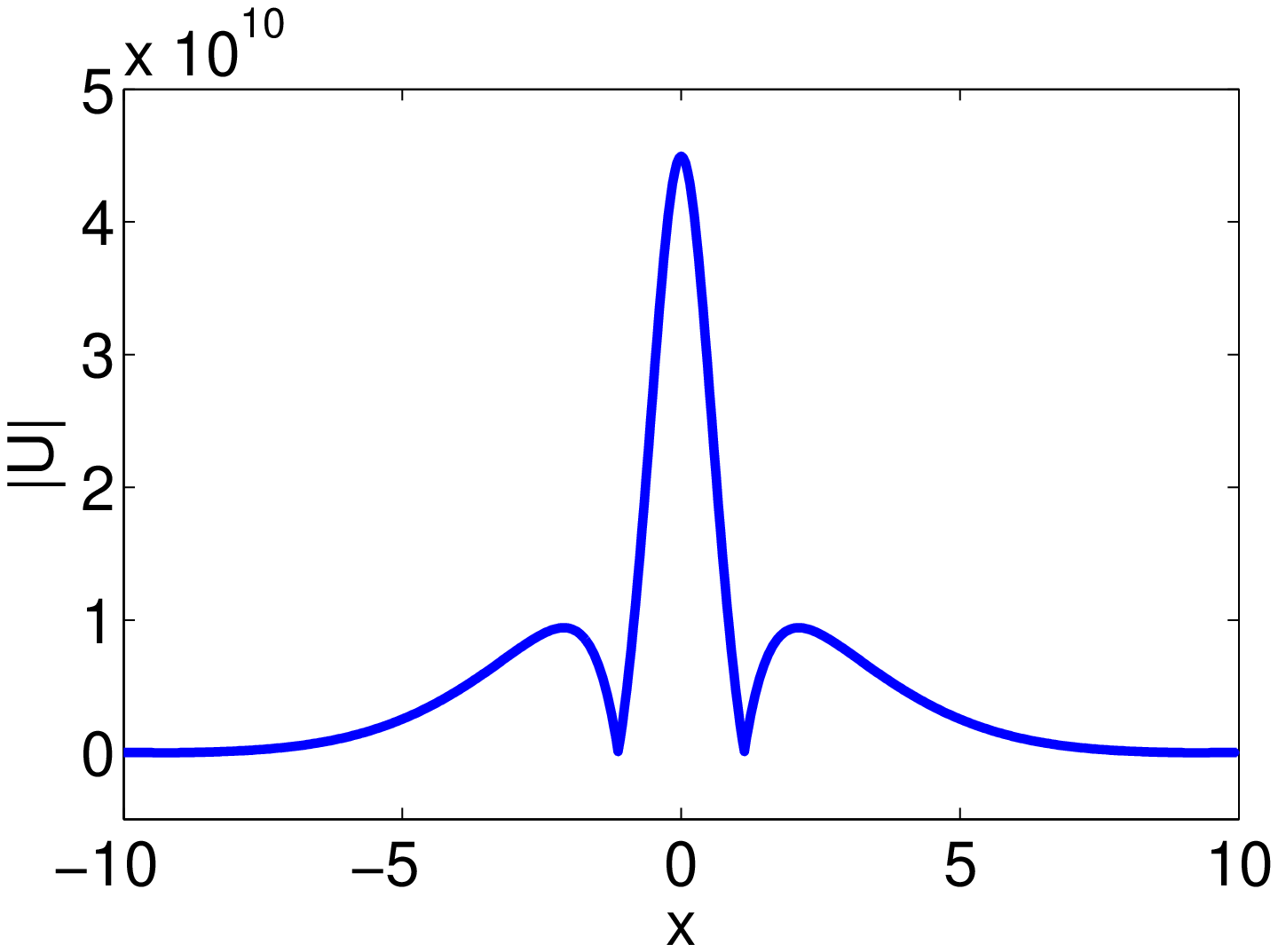}
\end{minipage}
\end{figure}
\begin{figure}[h!bt]

\begin{minipage}[t]{0.44\linewidth}
\includegraphics[width=2.5in]{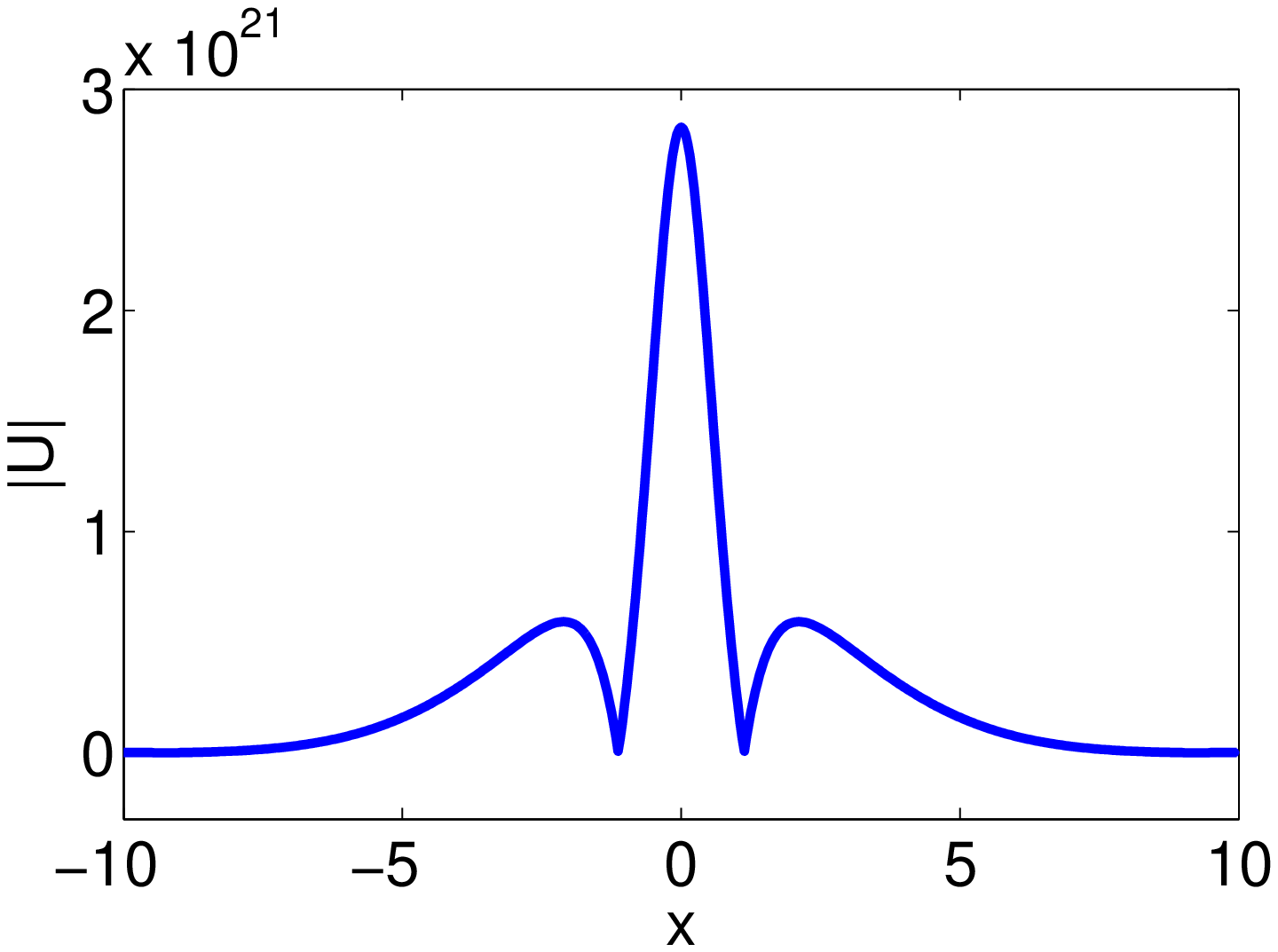}
\end{minipage}
\hspace*{30pt}
\begin{minipage}[t]{0.44\linewidth}
\hspace*{-5pt}
\includegraphics[width=2.5in]{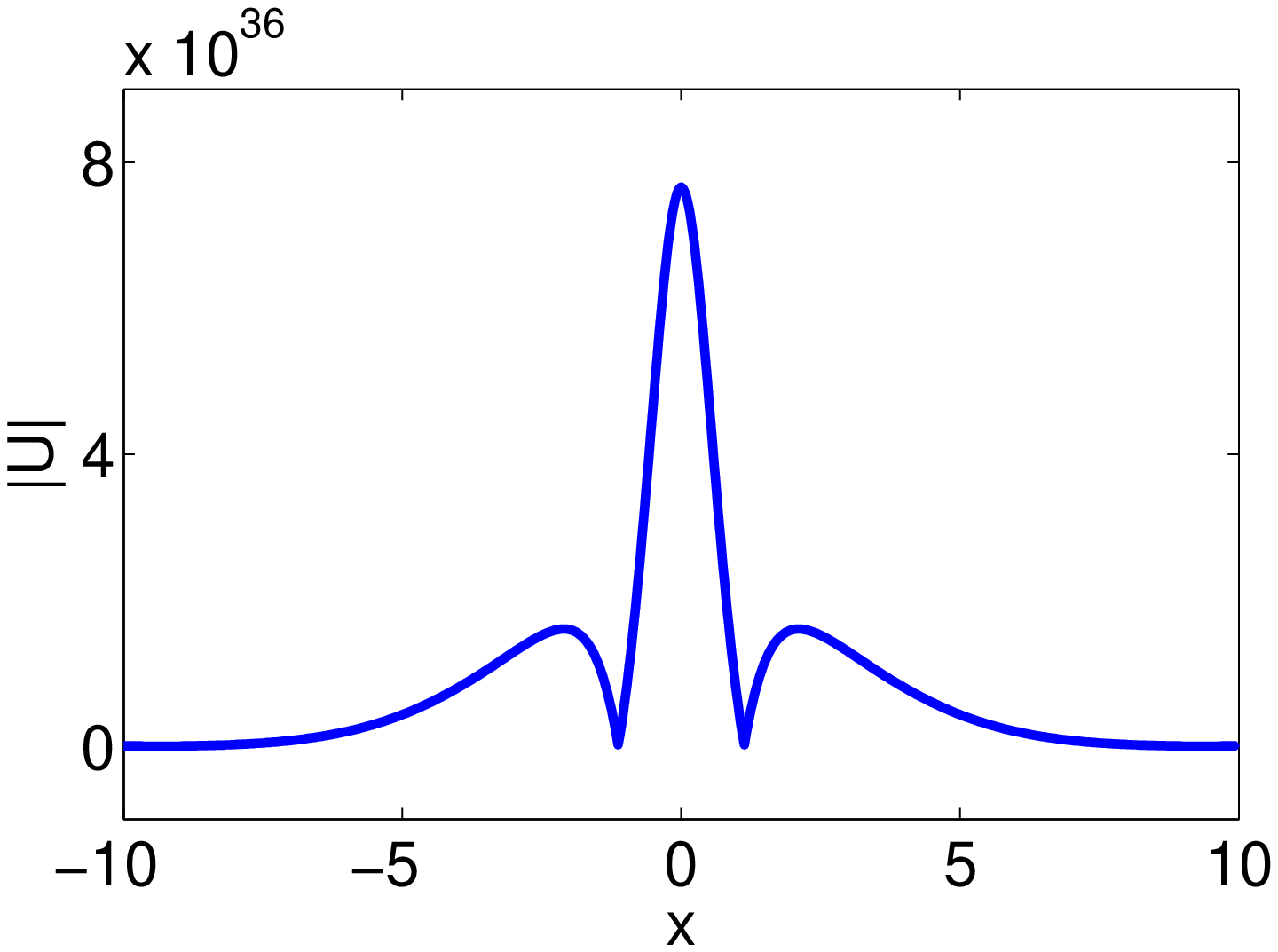}
\end{minipage}
\caption{\small{ The blow-up profile of the HBq equation near the blow-up time
at {$t=3.710$, $t=3.711$, $t=3.7115$, $t=3.712$ for quadratic nonlinearity with $\eta_1=\eta_2=1$. }}}
\end{figure}

\par
To investigate how the blow-up time depends on the coefficient of extra dispersion  term $\eta_2$ with the fixed value $\eta_1=1$, we perform some numerical experiments. The variation of the blow-up time as $\eta_2\rightarrow 0^+$ is presented in Figure 6. As it is seen from the figure,  less dispersion gives rise to earlier blow-up time.
\begin{figure}
\begin{minipage}[t]{0.44\linewidth}
\hspace*{-80pt}
\includegraphics[width=4in]{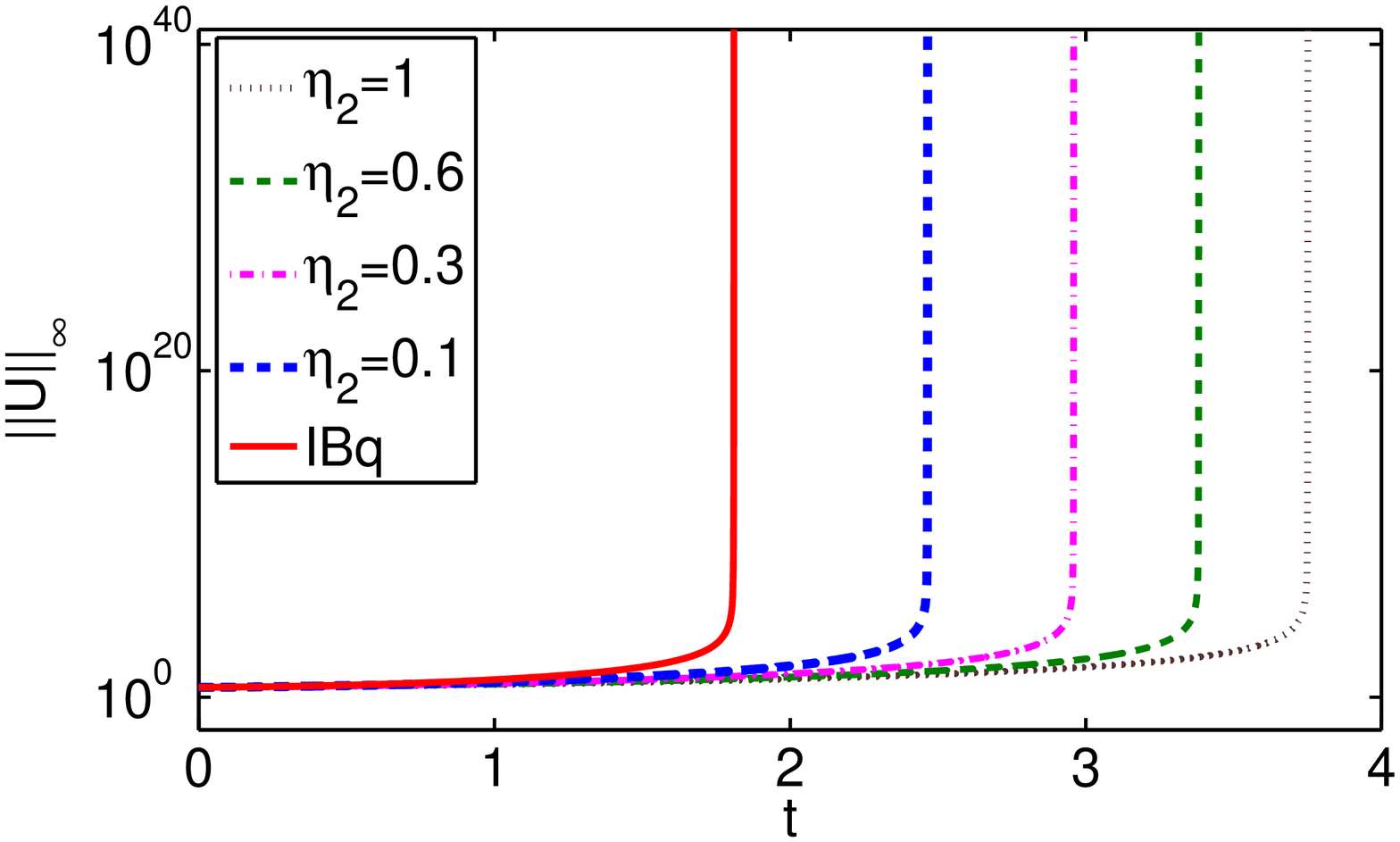}
\end{minipage}
\hspace*{-13pt}
\begin{minipage}[t]{0.44\linewidth}
\includegraphics[width=4.1in]{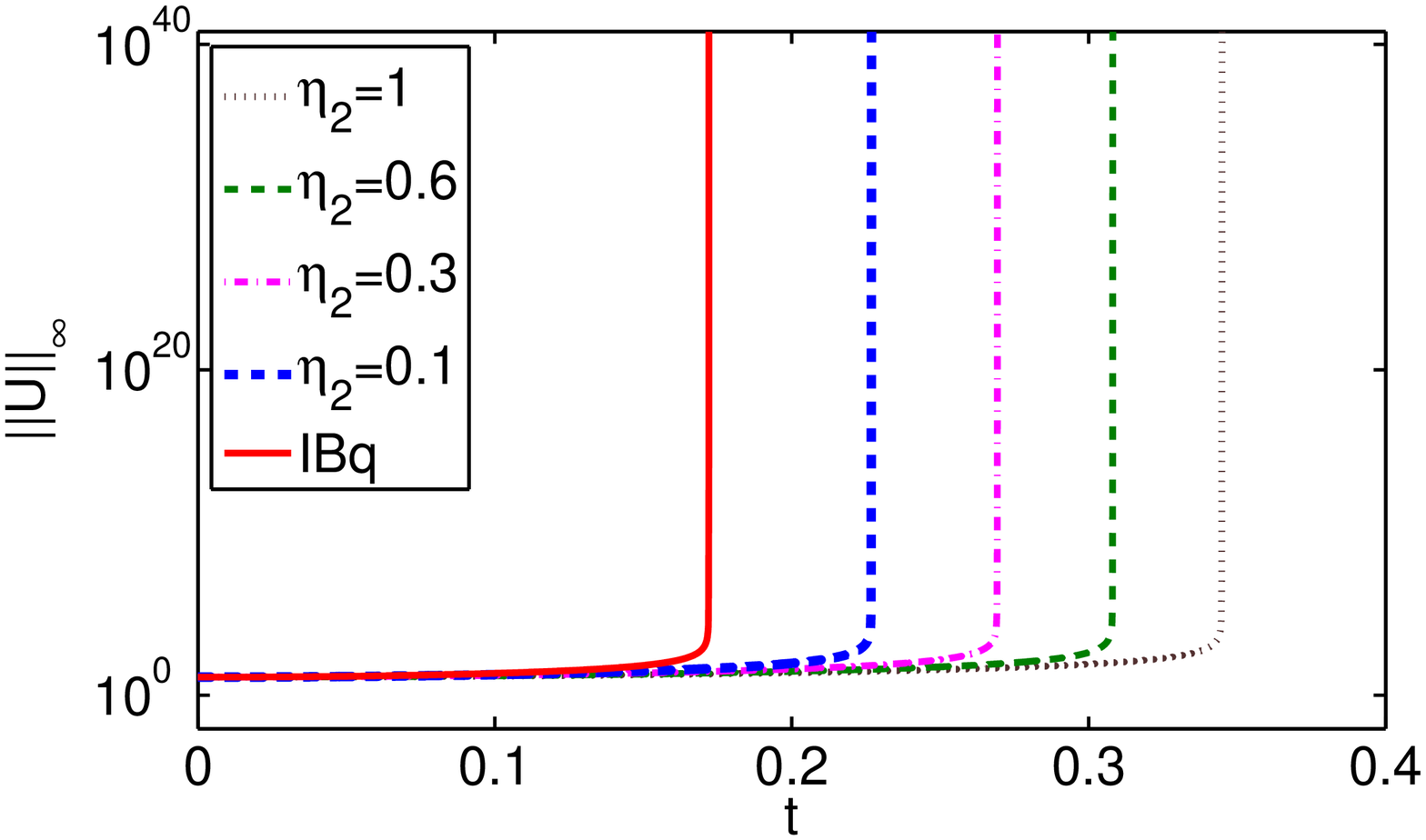}
\end{minipage}

\vspace*{-30pt}
\caption{\small{ The variation of the blow-up time as $\eta_2\rightarrow 0^+$ for quadratic nonlinearity (left panel) and  cubic nonlinearity (right panel).}}
\end{figure}
\par
To investigate how the blow-up time depends on  the power $p$ of  nonlinear term, we perform some numerical experiments.
We use the initial data \eqref{inamp} for $f(u)= u^p, ~p=2, 4, 6$  and \eqref{inamp2} for  $f(u)=-u^p,~p=3, 5$. We set the parameters $\eta_1=\eta_2=1$ in \eqref{hbq1}.
We note that it is possible to find some appropriate values for  $\mu$ to satisfy the condition \eqref{eszmu} for all these cases. Moreover, $I_3(0)$ also
remains negative  for increasing $p$ with the initial data \eqref{inamp} and \eqref{inamp2}. Figure 7 shows the blow-up time with respect to changing $p$.
It can be observed that the blow-up time is decreasing with increasing powers of nonlinearity.
\begin{figure}
\begin{minipage}[t]{0.44\linewidth}
\hspace*{-80pt}
\includegraphics[width=4in]{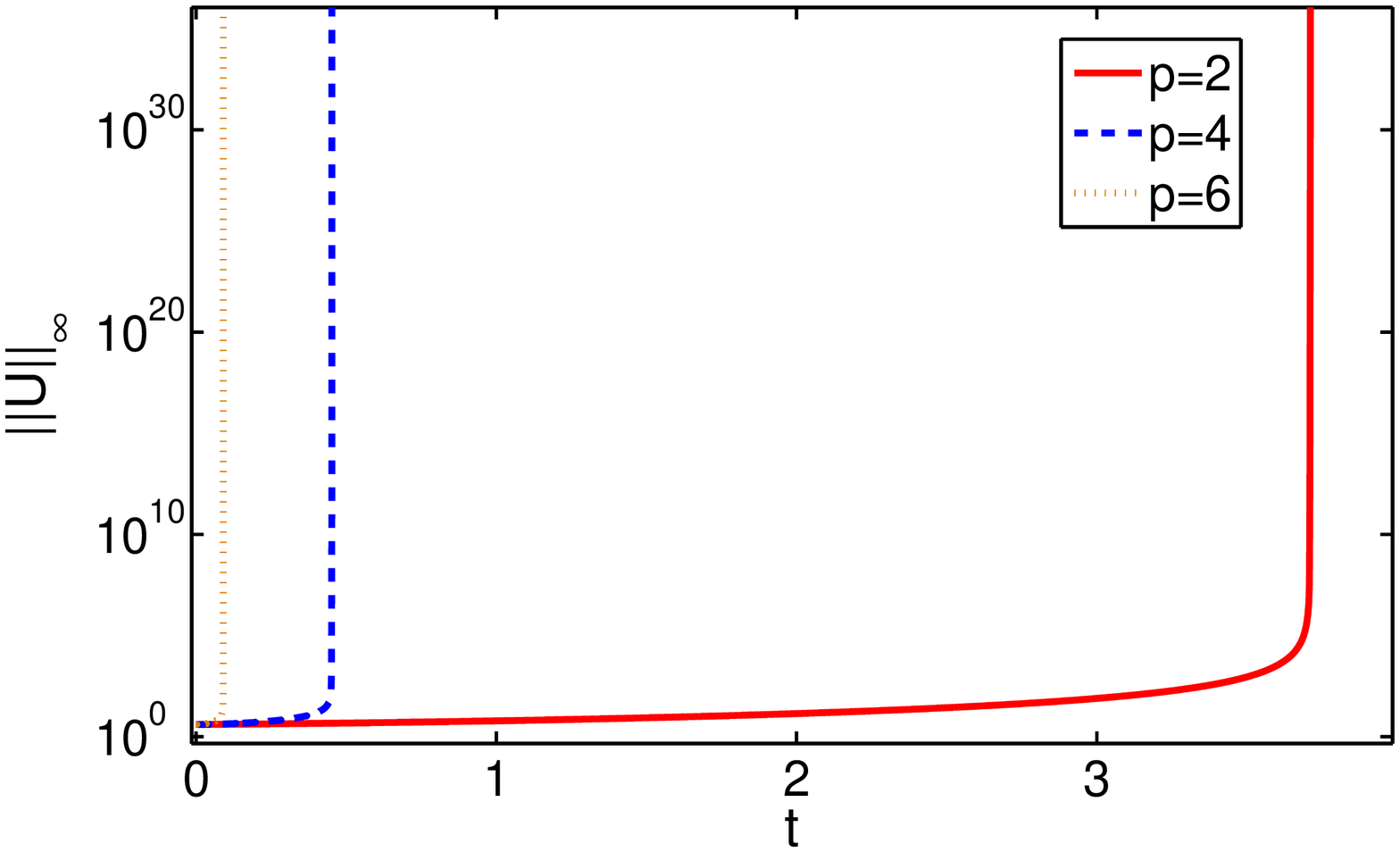}
\end{minipage}
\hspace*{10pt}
\begin{minipage}[t]{0.44\linewidth}
\includegraphics[width=4in]{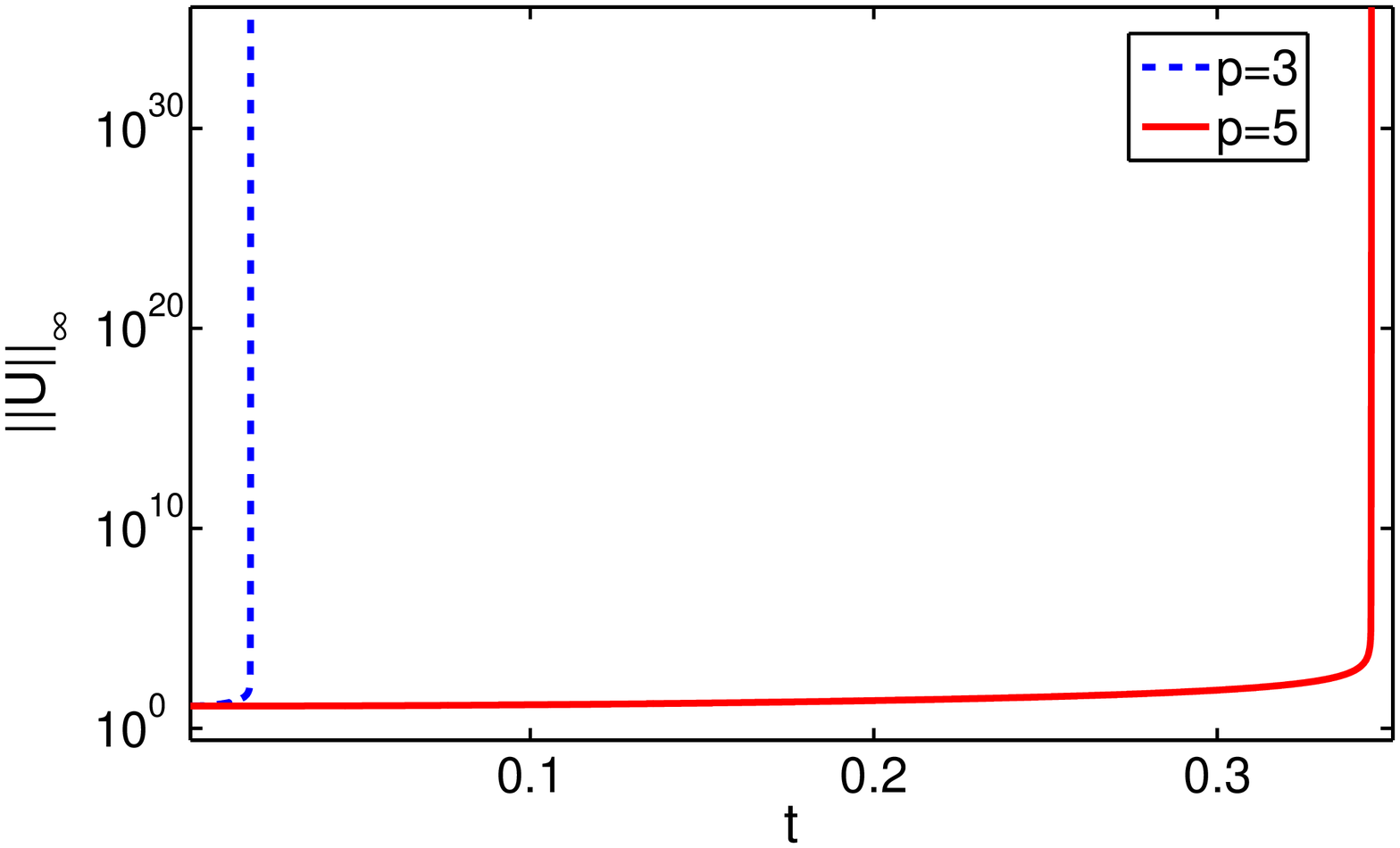}
\end{minipage}

\vspace*{-30pt}
\caption{\small{The blow-up time with respect to changing $p$.}}
\end{figure}

\vspace{30pt}
\noindent \textbf{Acknowledgement:} This work has been supported by the Scientific and Technological Research Council of Turkey
(TUBITAK) under the project MFAG-113F114. The authors  gratefully  acknowledge  to the anonymous reviewers for the constructive comments and valuable suggestions which improved the first draft of paper.


\bibliographystyle{elsart-num-sort}
\bibliography{kaynak}
\end{document}